\documentclass[a4paper,12pt]{amsart}
\usepackage{amssymb}
\usepackage{amscd}
\usepackage{amsthm}
\usepackage{amsmath}
\usepackage[all]{xy}
\usepackage{extarrows}
\usepackage{rotating}
\usepackage{tikz-cd}
\usepackage{leftidx}
\usepackage{bm}
\usepackage{dynkin-diagrams}
\usepackage{mathrsfs}
\usepackage{rotating}
\usepackage{upgreek}

\usepackage[pagebackref=true]{hyperref}
\renewcommand*{\backref}[1]{}
\renewcommand*{\backrefalt}[4]{[{\tiny%
		\ifcase #1 Not cited.%
		\or Cited on page~#2.%
		\else Cited on pages #2.%
		\fi%
	}]}

\renewenvironment{proof}{{ \textbf{Proof}.}}{\qed}
\newtheorem{Thm}{Theorem}[section]
\newtheorem{Lem}[Thm]{Lemma}
\newtheorem{Def}[Thm]{Definition}
\newtheorem{Cor}[Thm]{Corollary}
\newtheorem{Prop}[Thm]{Proposition}
\newtheorem{Ex1}[Thm]{Example}
\newtheorem{Rem1}[Thm]{Remark}

\newtheorem{assumption}{Assumption}

\newcommand{\cone}{\mathrm{Cone}}
\newcommand{\Hom}{\mathrm{Hom}}

\newcommand{\add}{\mathrm{add}}
\newcommand{\per}{\mathrm{per}}
\newcommand{\Ext}{\mathrm{Ext}}
\newcommand{\End}{\mathrm{End}}
\newcommand{\obj}{\mathrm{obj}}
\newcommand{\Mod}{\mathrm{Mod}}

\newcommand{\bijar}[1][]{%
	\ar[#1]
	\ar@<0.7ex>@{}[#1]|-*[@]{\sim}}

\newcommand{\pvd}{\mathrm{pvd}}

\newcommand{\HH}{H\!H}
\newcommand{\HC}{H\!C}
\newcommand{\HN}{H\!N}

\newcommand{\bL}{\mathbf{L}}
\newcommand{\charc}{\mathrm{char}}
\newcommand{\car}{\mathrm{Card}}

\newenvironment{Rem}{\begin{Rem1}\rm}{\end{Rem1}}
\newenvironment{Ex}{\begin{Ex1}\rm}{\end{Ex1}}

\usepackage{geometry}
\newcommand{\spa}{\mathrm{span}}
\newcommand{\Si}{\Sigma}

\newcommand{\ten}{\otimes}

\newcommand{\lten}{\overset{\mathbf{L}}{\ten}}

\newcommand{\cof}{\mathrm{cof}}
\newcommand{\RHom}{\mathrm{\mathbf{R}Hom}}
\newcommand{\ra}{\rightarrow}

\newcommand{\iso}{\xrightarrow{_\sim}}

\newcommand{\id}{\mathbf{1}}

\newcommand{\bmgamma}{\Gamma}
\newcommand{\pr}{\mathrm{pr}}

\newcommand{\relGammabf}{\bm\Gamma}
\newcommand{\overGamma}{\overline{\bmgamma}}
\newcommand{\Gr}{\mathrm{Gr}}
\renewcommand{\mod}{\mathrm{mod}}

\newcommand{\ind}{\mathrm{ind}}
\newcommand{\copr}{\mathrm{copr}}

\newcommand{\bmhom}{\mathrm{\textbf{Hom}}}
\newcommand{\ca}{{\mathcal A}}
\newcommand{\cb}{{\mathcal B}}
\newcommand{\cc}{{\mathcal C}}
\newcommand{\cd}{{\mathcal D}}

\newcommand{\ch}{{\mathcal H}}

\newcommand{\ck}{{\mathcal K}}

\newcommand{\cm}{{\mathcal M}}

\newcommand{\co}{{\mathcal O}}
\newcommand{\cp}{{\mathcal P}}

\newcommand{\ccr}{{\mathcal R}}
\newcommand{\cs}{{\mathcal S}}
\newcommand{\ct}{{\mathcal T}}

\newcommand{\cw}{{\mathcal W}}

\newcommand{\cy}{{\mathcal Y}}
\newcommand{\cz}{{\mathcal Z}}

\begin{document}
	
	\title{Group actions on relative cluster categories and Higgs categories}
	
	

	\author{Yilin WU}
	\address{
University of Luxembourg\\                           
6, avenue de la Fonte, L-4364\\
Esch-sur-Alzette\\            
Luxembourg
}
\email{wuyilinmath@gmail.com
}
	
	
	\dedicatory{}
	
	
	\begin{abstract}
	Let $G$ be a finite group acting on an ice quiver with potential $(Q, F, W)$. We construct the corresponding $G$-equivariant relative cluster category and  $G$-equivariant Higgs category, extending the work of Demonet. Using the orbit mutations on the set of $G$-stable cluster-tilting objects of the Higgs category and an appropriate cluster character, we can link these data to a skew-symmetrizable cluster algebra with coefficients. As a specific example, this provides an additive categorification for cluster algebras with principal coefficients in the non-simply laced case.
	\end{abstract}
	
	\maketitle
	\tableofcontents

	\section{Introduction}
	
Cluster algebras were invented by Fomin and Zelevinsky around 2001 \cite{FZ2002} in order to create a combinatorial framework for the study of canonical bases  in quantum groups and the study of total positivity in algebraic groups.  By construction, a cluster algebra is a commutative ring endowed with distinguished generators (cluster variables) grouped in subsets of the same cardinality (clusters). Since the combinatorics of cluster algebras is complicated, it is useful to model them categorically, so that more conceptual tools become available. For a cluster algebra $\ca$ defined by an acyclic quiver $Q$, Buan–Marsh–Reineke–Reiten–Todorov \cite{BMMRRT} introduced the cluster category $\cc_{Q}$. In order to generalize the representation-theoretic approach to cluster algebras from acyclic quivers
to quivers with oriented cycles, Derksen–Weyman–Zelevinsky \cite{DWZ2008,DWZ2010} extended the mutation operation from quivers to quivers with potential and their representations. In the case where the quiver with potential is Jacobi-finite, Amiot \cite{Amoit2009} generalized the construction of the cluster category \cite{BMMRRT}. The cluster character constructed by Palu in \cite{paluClusterCharacters2Calabi2008} induces a bijection \cite{CKLP2013} from the isoclasses of the reachable rigid indecomposables of the (generalized) cluster category to the cluster variables of the associated cluster algebra. Plamondon \cite{Plamondon2011} generalized Amiot’s and Palu’s constructions to arbitrary quivers with potential.

Cluster algebras with coefficients are important since they appear in
nature as coordinate algebras of varieties like Grassmannians, double Bruhat cells, unipotent cells, etc.
The work of Geiss-Leclerc-Schr\"{o}er often yields Frobenius exact categories which allow us to categorify such cluster algebras. In \cite{Wu2023}, we have generalized the construction of (higher) cluster categories by Claire Amiot and by
Lingyan Guo to the relative context. We proved the existence of an $n$-cluster tilting object in a Frobenius extriangulated category, namely the Higgs category, which is stably $n$-Calabi–Yau and Hom-finite, arising from a left $(n+1)$-Calabi--Yau morphism. Higgs categories generalize the Frobenius categories used by Geiss–Leclerc–Schr\"{o}er but are no longer exact categories in the sense of Quillen in general. They serve to categorify cluster algebras with non invertible coefficents whereas relative cluster categories serve to categorify their localizations at the coefficients.

However, the theory of Fomin–Zelevinsky accommodates more general seeds defined by skew-symmetrizable matrices. For instance, cluster algebras associated with Lie groups of types $B$, $C$, $F$, and $G$ are skew-symmetrizable, as are cluster algebras with principal coefficients of these types.

The aim of this paper is to generalize the results of Demonet \cite{LD2011} to the setting of ice quivers with potentials. By using $G$-orbit mutations on the set of $G$-stable cluster-tilting objects in the Higgs category and an appropriate cluster character, we categorify certain skew-symmetrizable cluster algebras with coefficients, including, in particular, cluster algebras with principal coefficients in the non-simply-laced case. We remark that Azzurra Ciliberti has recently provided a categorification of cluster algebras of type $B$ and $C$ with principal coefficients via symmetric quivers \cite{Ciliberti2025}.

The structure of the paper is as follows. In Section \ref{Section2}, we give some background material on dg algebras and relative Calabi-Yau completions. Let $k$ be field and $f\colon B\ra A$ a dg morphism between smooth dg algebras over $k$. Let $G$ be a finite group acting on $f$, i.e. $G$ acts on $B$ and $A$, and $f$ is $G$-equivariant. We assume that $\charc(k)\nmid\car(G)$. Let $n$ be a positive integer and $\xi$ an element of $\HH_{n-2}(f)^{G}\subseteq\HH_{n-2}(f)$. Denote by $\xi_{B}$ the element in $\HH_{n-3}(B)$ induced by $\xi$ under the map $\HH_{n-2}(f)\ra \HH_{n-3}(B)$. 

In Section \ref{Section3}, we show that the deformed relative 3-Calabi--Yau completion $$ \tilde{f}\colon\bm\Pi_{n-1}(B,\xi_{B})\ra\bm\Pi_{n}(A,B,\xi) $$ has a natural $G$-action, i.e. $G$ acts on $\bm\Pi_{n-1}(B,\xi_{B})$ and $\bm\Pi_{n}(A,B,\xi)$ and $\tilde{f}$ is $G$-equivariant. Moreover, if $\xi$ lifts to $\HN_{n-2}(f)$, then the associated dg functor between skew group dg algebras
$$\tilde{f}*G\colon\bm\Pi_{n-1}(B,\xi_{B})*G\ra\bm\Pi_{n}(A,B,\xi)*G$$ has a canonical left $n$-Calabi--Yau structure, see Theorem \ref{Thm:skew CY completion}. Let $(Q,F)$ be a graded ice quiver and $W$ a homogeneous potential on $Q$ of degree $3-n$. 
Suppose that Assumption \ref{ass1} in Section~\ref{Section 3.5} is satisfied. Then there exists a graded ice quiver with potential $(Q_{G},F_{G},W_{G})$ such that we have the following commutative diagram (Proposition \ref{Prop: Lemeur-Morita})
\[
\begin{tikzcd}
	\bm{\Pi}_{n-1}(F_{G})\arrow[r]\arrow[d]&\bmgamma_{n}(Q_{G},F_{G},W_G)\arrow[d]\\
	\bm{\Pi}_{n-1}(F)*G\arrow[r]&\bmgamma_{n}(Q,F,W)*G,
\end{tikzcd}
\]
Where $\bm{\Pi}_{n-1}(F)$ is the $(n-1)$-Calabi--Yau completion of $kF$, see Definition \ref{Def: derived preprojective algebra} and $\bmgamma_{n}(Q,F,W)$ is the $n$-dimensional relative Ginzburg dg algebra, see Definition \ref{Def:Relative Ginzburg algebra}. And
the restriction-of-scalars functors induce equivalences of triangulated categories
\[ \cd(\bmgamma_{n}(Q,F,W)*G)\simeq\cd(\bmgamma_{n}(Q_{G},F_{G},W_G)) \]
and 
\[ \cd(\bm{\Pi}_{n-1}(F)*G)\simeq\cd(\bm{\Pi}_{n-1}(F_{G})) .\]

In Section \ref{Section4:Equivariant categories}, we recall some facts on separable monads and the construction of the comparison functor associated to an adjoint pair. Let $A$ be a dg $k$-algebra and $G$ a finite group acting on $A$ by dg algebra automorphisms. Then $G$ naturally acts on $\per(A)$ and $\pvd(A)$. Let $\per(A)^G$ and $\pvd(A)^G$ be the corresponding categories of $G$-equivariant objects. Then we show the following triangle equivalences, see Proposition \ref{Prop:equ of cats}
$$ \per(A*G)\iso\per(A)^{G}\quad\pvd(A*G)\iso\pvd(A)^{G}.$$ In Section \ref{Section5}, we present the main construction of the $G$-equivariant relative cluster category and the $G$-equivariant Higgs category. Let $(Q,F)$ be an ice quiver and $W$ a potential on $Q$ of degree 0. Denote by $ \bm{\Gamma} $ the corresponding relative Ginzburg algebra, i.e. the 3-dimensional relative Ginzburg dg algebra. Let $\overline{Q}$ be the quiver obtained from $Q$ by deleting all vertices in $F$ and all arrows incident with vertices in $F$. Let $\overline{W}$ be the potential on $Q$ obtaining by deleting all cycles passing through vertices of $F$ in $W$. 

Let $\overGamma$ be the Ginzburg algebra associated with $(\overline{Q},\overline{W})$. Let $ e=\sum_{i\in F}e_{i} $ be the idempotent associated with the set of frozen vertices. Suppose that Assumptions~\ref{ass1} and~\ref{ass2} (see Sections~\ref{Section 3.5} and~\ref{Section5}, respectively) are satisfied. By our assumption, $\overGamma$ has an induced $G$-action. Then the \emph{$G$-equivariant relative cluster category} $\cc(\bmgamma*G)$ is defined as the idempotent completion of the Verdier quotient of triangulated categories $$ \mathrm{per}(\bm\Gamma*G)/\mathrm{pvd}(\overGamma*G),$$
where we view $\mathrm{pvd}(\overGamma*G)$ as a triangulated subcategory of $\per(\bmgamma*G)$ through $$\mathrm{pvd}(\overGamma*G)\simeq(\pvd_{e}(\bmgamma))^{G}\hookrightarrow\per(\bmgamma)^{G}\simeq\per(\bmgamma*G).$$
Then Proposition \ref{Prop:skew-relative-cluster} shows that the $\cc(\bmgamma*G)$ is triangle equivalent to the equivariant category $\cc(\bmgamma)^{G}=(\per\bmgamma/\pvd_{e}(\bmgamma))^{G}$. By Definition \ref{Def: Q_{G} and F_{G}} and Corollary \ref{Cor: stable derived Morita}, there exists an ice quiver with potential $(Q_{G},F_{G},W_{G})$ such that $ \cc(\bmgamma*G) $ is equivalent to the relative cluster category $\cc(\bmgamma_{G})$ associated with $(Q_{G},F_{G},W_{G})$. Then we define the \emph{G-equivariant Higgs category} to be $\ch(\bmgamma_G)$, i.e. the Higgs category associated with $(Q_{G},F_{G},W_{G})$
\[
\ch(\bmgamma_G)=\{ X \in \cc(\bmgamma_G) \mid \operatorname{Hom}_{\cc(\bmgamma)}(X, \Sigma^{>0} \mathcal{P}') = 0 = \operatorname{Hom}_{\cc(\bmgamma))}(\Sigma^{<0} \mathcal{P}', X) \},
\]
where $\cp'=\add(e_{G}\bmgamma_G)$, i.e. the additive subcategory of $\cc(\bmgamma_G)$ generated by $e_{G}\bmgamma_G=\displaystyle\oplus_{k\in F_{G}}e_{k}\bmgamma_{G}$. We show that the triangle equivalence $\cc(\bmgamma*G)\iso\cc(\bmgamma)^{G}$ induces an equivalence of Frobenius extriangulated categories
$$ \ch(\bmgamma_G)\iso\ch(\bmgamma)^{G},$$
where the category of projective-injective objects of $\ch^{G}$ is $\cp^{G}=\add(e\bmgamma)^{G}$ and $\bmgamma[G]=(\oplus_{h\in G}\,^{h}\bmgamma,\mathrm{Id})$ is a canonical cluster-tilting object of $\ch^{G}$, see Theorem \ref{Prop: skew-Higgs}.

In the final section, we construct an appropriate cluster character for the Higgs category, enabling us to connect this data to a skew-symmetrizable cluster algebra with coefficients. As a specific example, the Higgs category together with the set of $G$-stable cluster tilting objects provides an additive categorification of cluster algebras with principal coefficients in the non-simply-laced case.

\section*{Acknowledgments}


The author is very grateful to Bernhard Keller for many helpful comments and suggestions. He is also grateful to Xiao-Wu Chen and Guodong Zhou for their constant support and encouragement. He would also like to thank Peigen Cao, Xiaofa Chen, and Dong Yang for stimulating discussions. Furthermore, the author thanks Sarah Scherotzke for her support and the excellent working conditions during his postdoctoral stay at the University of Luxembourg. Finally, the author is grateful to the referees for their careful work and constructive suggestions. This work is supported by the China Postdoctoral Science Foundation (2023M733405, 2024T170871) and the Fundamental Research Funds for the Central Universities.

%
%
%

\section{Preliminaries}\label{Section2}


Let $ k $ be a field. A \emph{differential graded $ k $-algebra} (or simply dg $ k $-algebra) is a graded $ k $-algebra $ A=\bigoplus_{n\in\mathbb{Z}}A^{n} $ equipped with a $ k $-linear homogeneous map $ d_{A}:A\ra A $ of degree 1 such that $ d_{A}^{2}=0 $ and the graded Leibniz rule $ d_{A}(ab)=(d_{A}a)b+(-1)^{n}ad_{A}(b) $ holds, where $ a\in A^{n} $ and $ b\in A $. The map $ d_{A} $ is called the \emph{differential} of $ A $. We can view an ordinary $ k $-algebra as a dg $ k $-algebra concentrated in degree $ 0 $ whose differential is trivial. A graded $ k $-algebra can be viewed as a dg $ k $-algebra with the zero differential. Let $A$ be a dg $k$-algebra.

\begin{Def}\rm
	A \emph{right dg module} over $ A $ is a graded right $ A $-module $ M=\bigoplus_{n\in\mathbb{Z}}M^{n} $ equipped with a $ k $-linear homogeneous map $ d_{M}:M\ra M $ of degree 1 such that $ d_{M}^{2}=0 $ and the graded Leibniz rule $$ d_{M}(ma)=d_{M}(m)a+(-1)^{n}md_{A}(a) $$ holds for all $ m\in M^{n} $ and $ a\in A $. The map $ d_{M} $ is called the \emph{differential} of $ M $.
\end{Def}

Given two dg $ A $-modules $ M $ and $ N $, we define the \emph{morphism complex} to be the graded $ k $-vector space $ \ch om_{A}(M,N) $ whose $ i $-th component $ \ch om^{i}_{A}(M,N) $ is the subspace of the product $ \prod_{j\in\mathbb{Z}}\Hom_{k}(M^{j},N^{j+i}) $ consisting of morphisms $ f $ such that $ f(ma)=f(m)a $ for all $ m $ in $ M $ and all $ a $ in $ A $, together with the differential $ d $ given by
$$ d(f)=f\circ d_{M}-(-1)^{|f|}d_{N}\circ f $$ for a homogeneous morphism $ f $ of degree $ |f| $.

The category $ \cc(A) $ of dg $ A $-modules is the category whose objects are the right dg $ A $-modules, and whose morphisms are the 0-cycles of the morphism complexes. It is an abelian category and a Frobenius category for the conflations which are split exact as sequences of graded $ A $-modules (see \cite{kellerDerivingDGCategories1994a}). Its stable category $ \ch(A) $ is called the \emph{homotopy category} of right dg $ A $-modules. It can be equivalently defined as the category whose objects are the dg $ A $-modules and whose morphism spaces are the 0-th cohomology groups of the morphism complexes.

The homotopy category $ \ch(A) $ is a triangulated category whose suspension functor $ \Si $ is the shift of dg modules $ M\mapsto \Si M $. The \emph{derived category} $ \cd(A) $ of dg $ A $-modules is the localization of $ \ch(A) $ at the full subcategory of acyclic dg $ A $-modules. 

A dg $ A $-module $ P $ is \emph{cofibrant} if, for every surjective quasi-isomorphism $ L\ra M $, every morphism $ P\ra M $ factors through $ L $. For example, the dg algebra $ A $ considered as a right module over itself is cofibrant. A dg $ A $-module $ I $ is \emph{fibrant} if, for every injective quasi-isomorphism $ L\ra M $, every morphism $ L\ra I $ extends to $ M $.

\begin{Prop}\cite{kellerDerivingDGCategories1994a}
	\begin{itemize}
		\item[a)] For each dg $ A $-module $ M $, there is a quasi-isomorphism $ \mathbf{p}M\ra M $ with cofibrant $ \mathbf{p}M $ and a quasi-isomorphism $ M\ra \mathbf{i}M $ with fibrant $ \mathbf{i}M $.
		\item[b)] The projection functor $ \ch(A)\ra \cd(A) $ admits a fully faithful left adjoint given by $ M\mapsto\mathbf{p}M$ and a fully faithful right adjoint given by $ M\mapsto\mathbf{i}M $.
	\end{itemize}
\end{Prop}
We call $ \mathbf{p}M\ra M $ a \emph{cofibrant resolution} of $ M $ and $ M\ra \mathbf{i}M $ a \emph{fibrant resolution} of $ M $. We can compute morphisms in $ \cd(A) $ via
$$ \cd(A)(L,M)\simeq\ch(A)(\mathbf{p}L,M)=H^{0}(\ch om_{A}(\mathbf{p}L,M)) .$$

The \emph{perfect derived category} $ \per A $ is the smallest full subcategory of $ \cd(A) $ containing $ A $ which is stable under taking shifts, extensions and direct summands. An object $ M $ of $ \cd(A) $ belongs to $ \per(A) $ if and only if it is \emph{compact}, i.e. the functor $ \Hom_{\cd(A)}(M,-) $ commutes with arbitrary (set-indexed) direct sums (see~\cite[Section 5]{kellerDerivingDGCategories1994a}). The \emph{perfectly valued derived category} $ \pvd(A) $ is the full subcategory of $ \cd(A) $ consisting of those dg $ A $-modules whose underlying dg $ k $-module is perfect, i.e. the total dimension $\sum_{i\in\mathbb{Z}}\dim H^{i}(A)$ is finite.

Let $ f:B\ra A $ be a morphism of dg algebras. Then $ f $ induces the restriction functor $ f_{*}: \cc(A)\ra\cc(B) $. It fits into the usual triple of adjoint functors $ (f^{*} , f_{*},f^{!}) $ between $ \cd(A) $ and $ \cd(B) $.

Denote by $A^{op}$ the opposite dg algebra of $A$. The \emph{enveloping dg algebra} $A^{e}$ of $A$ is defined as $A\ten_{k}A^{op}$ with product $(a\ten b)\cdot(c\ten f)\coloneqq (-1)^{\deg(b)\cdot(\deg(c)+\deg(f))}ac\ten fb$ for all (homogeneous) $a, b, c, d\in A$. The category of dg $A$-bimodules is equivalent to the category of right $A^{e}$-modules. 

In particular, $A$ is a dg $A^{e}$-module for the following structure map
$$A\ten A^{e}\ra A,\,a\ten(b\ten c)\mapsto(-1)^{\deg(c)\cdot(\deg(b)+\deg(a))}cba.$$ We say that $A$ is \emph{(homologically) smooth} if $A$ belongs to $\per(A^{e})$ and $A$ is \emph{proper} if $A\in\pvd(A)$, i.e. the total dimension $\sum_{l\in\mathbb{Z}}\dim_{l} H^{l}(A)$ is finite.

\begin{Def}\rm\label{Def:bimodule dual}
	For a dg $A^{e}$-module $M$, the \emph{derived bimodule dual} $M^{\vee}$ of $M$ is defined as $\RHom_{A^{e}}(M,A^{e})$.
	
\end{Def}

\begin{Def}\rm
	A \emph{graded quiver $Q$} consists of a finite set $Q_{0}$ and a graded set $Q_{1}$ (i.e., a set $Q_{1}$ together with a degree
	map $|\,\,|\colon Q_{1}\ra\mathbb{Z}$), together with two maps $s,t\colon Q_{1}\ra Q_{0}$. We will often simply say that $Q$ is a graded quiver over $\co$. The associated graded path algebra is defined as the graded tensor algebra $T_{kQ_{0}}(kQ_{1})$. 
\end{Def}


A dg algebra $A$ is said to be \emph{semi-free} if its underlying graded vector space (forgetting the differential) is isomorphic to $T_{kQ_{0}}(kQ_{1})$ for some graded quiver $(Q_{0},Q_{1})$.

A semi-free dg algebra $A=T_{kQ_{0}}(kQ_{1})$ is said to be \emph{cellular} if the set $Q_{1}$ has a filtration $Q'_{1}\subset Q'_{2}\subset\cdots $ such
that $Q_{1}=\bigcup Q'_{i}$ and that, for each $f\in Q'_{i}
, i=1, 2,\cdots$, the differential $df$ lies in the $T_{kQ_{0}}(kQ'_{i-1})\subset T_{kQ_{0}}(kQ_{1})$. It is said to be \emph{finitely cellular} if $Q_{1}$ is also finite.

\subsection{Hochschild and cyclic homology}

Let $ \Lambda $ be the dg algebra generated by an indeterminate $ \epsilon $ of cohomological degree $ -1 $ with $ \epsilon^{2}=0 $ and $ d\epsilon=0 $. The underlying complex of $ \Lambda $ is
$$ \cdots\rightarrow k\epsilon\rightarrow k\rightarrow0\ra\cdots. $$
Then a \emph{mixed complex} over $ k $ is a dg right $ \Lambda $-module whose underlying dg $ k $-module is $ (M,b) $ and where $ \epsilon $ acts by a closed endomorphism $ B $. Suppose that $ M=(M,b,B) $ is a mixed complex. Then the \emph{shifted mixed complex} $ \Si M $ is the mixed complex such that $ (\Si M)^{p}=M^{p-1} $ for all $ p $, $ b_{\Si M}=-b $ and $ B_{\Si M}=B $. Let $ f:M\rightarrow M' $ be a morphism of mixed complexes. Then the \emph{mapping cone} over $ f $ is the mixed complex
$$ \bigg(M'\oplus \Si M,\begin{bmatrix}
	b_{M'} & f \\
	0 & -b_{M}
\end{bmatrix},\begin{bmatrix}
	B_{M'} & 0 \\
	0 & B_{M}
\end{bmatrix}\bigg). $$
We denote by $ \cm ix $ the category of mixed complexes and by $ \cd\cm ix $ the derived category of the dg algebra $ \Lambda $.

Let $ A $ be  a dg $ k $-algebra. We associate a precyclic chain complex $ C(A) $ (see~\cite[Definition 2.5.1]{loday2013}) with $ A $ as follows: For each $ n\in\mathbb{N} $, its $ n $-th term is
$$ C_{n}(A)=A\ten_{k}A^{\ten_{k}^{n}} .$$
The degeneracy maps are given by
\begin{equation*}
	d_{i}(a_{n},\ldots,a_{i},a_{i-1},\ldots,a_{0})=\left\{
	\begin{aligned}
		(a_{n},\ldots,a_{i}a_{i-1},\ldots,a_{0})&&\text{if $ i>0 $,}\\
		(-1)^{n+\sigma}(a_{0}a_{n},\ldots,a_{1})&&\text{if $ i=0 $},
	\end{aligned}
	\right.
\end{equation*}
where $ \sigma=(\mathrm{deg}a_{0})(\mathrm{deg}a_{1}+\cdots+\mathrm{deg}a_{n-1}) $. The cyclic operator is given by
$$ t(a_{n-1},\ldots,a_{0})=(-1)^{n+\sigma}(a_{0},a_{n-1},a_{n-2},\cdots,a_{1}) .$$

Then the corresponding \emph{total complex} $ (H\!H(A),b) $ of $ (C(A),b=\sum_{i=0}^{n}(-1)^{i}d_{i}) $ is called the \emph{Hochschild complex} of $ A $ and $b$ is called the Hochschild differential of $A$. The \emph{Hochschild homology} of $ A $ is defined to be the cohomology of this complex. By~\cite[Proposition B.1]{Vandenbergh2015}, the Hochschild complex is quasi-isomorphic to $ A\lten_{A^{e}}A $ in $ \cd(k) $.

We associate a mixed complex $ (M(A),b,B) $ with this precyclic chain complex as follows: Consider the \emph{total complex} $ (H\!H(A),b’) $ of $ (C(A),b’=\sum_{i=0}^{n-1}(-1)^{i}d_{i}) $. The underlying dg module of $ M(A) $ is the mapping cone over $ (1-t) $ viewed as a morphism of complexes

$$ 1-t:(H\!H(A),b')\ra(H\!H(A),b) ,$$ where $ b=\sum_{i=0}^{n}(-1)^{i}d_{i} $ and $ b'=\sum_{i=0}^{n-1}(-1)^{i}d_{i} $.
Its underlying module is $ H\!H(A)\oplus H\!H(A) $; it is endowed with the grading whose $ n $-th component is $ H\!H(A)_{n}\oplus H\!H(A)_{n-1} $ and the differential is 
$$ \begin{bmatrix} b & 1-t \\ 0 & -b' \end{bmatrix} .$$
The operator $ B\colon M\ra M $ is
$$ \begin{bmatrix} 0 & 0 \\ N & 0 \end{bmatrix} ,$$ where $ N=\sum_{i=0}^{n}t^{i} $.

Let $ A' $ be an other dg $ k $-algebra. Let $ f $ be a morphism from $ A' $ to $ A $. Then $ f $ induces a canonical morphism between their Hochschild complexes
$$ \gamma_{f}\colon H\!H(A')\ra H\!H(A) .$$
\begin{Def}\rm
	The \emph{Hochschild homology} $ H\!H_{\bullet}(f) $ of $ f $ is the cohomology of the \emph{relative Hochschild complex} which is defined as follows
	$$ H\!H(f)=\cone(\gamma_{f}\colon HH(A')\ra HH(A)) .$$
\end{Def}

\begin{Rem}
	By definition, a closed element $ \xi=(s\xi_{A'},\xi_{A})\in \cone(\gamma_{f}\colon\HH(A')\ra \HH(A)) $ of degree $ -n $ consists of an element $ \xi_{A'} \in HH(B) $ of degree $ -n+1 $, together with an element $ \xi_{A}\in \HH(A) $ of degree $ -n $, such that $ b_{A'}(\xi_{A'})=0 $ and $ b_{A}(\xi_{A})+\gamma_{f}(\xi_{A'})=0 $, where $ b_{A'} $ and $ b_{A} $ are the Hochschild differentials of $ A' $ and $ A $ respectively.
\end{Rem}

\begin{Def}\rm
	The \emph{cyclic homology} $ H\!C_{\bullet}(A) $ of $ A $ is defined to be the cohomology of the \emph{cyclic chain complex} of $ A $
	$$ \HC(A)=M(A)\lten_{\Lambda}k .$$
	
	The \emph{negative cyclic homology} $ H\!N_{\bullet}(A) $ of $ A $ is defined to be the cohomology of the \emph{negative cyclic chain complex} of $ A $
	$$ \HN(A)=\RHom_{\Lambda}(k,M(A)). $$
\end{Def}

The augmentation morphism $ \Lambda\to k $ induces natural morphisms in $ \cd(k) $
$$ \HN(A)\rightarrow \HH(A)\rightarrow \HC(A).$$

The morphism $ f $ also induces a canonical morphism between their mixed complexes
$$ \gamma_{f}\colon M(A')\ra M(A) .$$
We denote by $ M(f) $ the mapping cone over $ \gamma_{f} $.
\begin{Def}\rm
	The \emph{cyclic homology} $ \HC_{\bullet}(f) $ of $ f\colon A'\rightarrow A $ is defined to be the cohomology of the \emph{cyclic chain complex group} of $ f $
	$$ \HC(f)=M(f)\lten_{\Lambda}k. $$
	The \emph{negative cyclic homology} $ \HN_{\bullet}(f) $ of $ f\colon B\rightarrow A $ is defined to be the cohomology of the \emph{negative cyclic chain complex} of $ f $
	$$ H\!N(f)=\RHom_{\Lambda}(k,M(f)). $$
\end{Def}

\subsection{Relative deformed Calabi--Yau completions}\label{subsection:relative CY completion}
Given a dg algebra $B$, let $(\mathrm{dga}_{k})_{B/}$ be the category of dg algebras under $B$. The forgetful functor $(\mathrm{dga}_{k})_{B/}\ra \cc(B^{e})$, sending a dg morphism $B\ra A$ to the $B$-bimodule $A$, has a left adjoint $T_{B}\colon\cc(B^{e})\ra (\mathrm{dga}_{k})_{B/}$, that can be described as follows$\colon$
\begin{itemize}
	\item[] Let $M$ be a $B$-bimodule. The \emph{tensor algebra} $T_{B}M$ is defined as $$T_{B}M=B\oplus M\oplus (M\ten_{B}M)\oplus\cdots.$$
	The dg structure on $T_{B}M$ is given by the differentials of $B$ and $M$ and the multiplication is given by the concatenation product and we have a canonical morphism $B\ra T_{B}M$ of dg algebras.
\end{itemize}
This is a Quillen adjunction and thus induces an adjunction between their
homotopy categories. We will denote by $\bL T_{B}$ the left derived functor.

Let $ f\colon B\ra A $ be a morphism (not necessarily unital) between smooth dg algebras. Let $ [\xi] $ be an element in $ H\!H_{n-2}(f) $. In this section, we recall the construction of the deformed relative $ n $-Calabi–Yau completion of $ f\colon B\ra A $ with respect to the Hochschild homology class $ \xi\in H\!H_{n-2}(f) $. For more details we refer the reader to \cite{WkY2016}.

The morphism $ f $ induces a morphism in $ \cd(A^{e}) $
$$ m_{f}\colon B\lten_{B^{e}}A^{e}\ra A .$$

After taking the derived bimodule dual (Definition \ref{Def:bimodule dual}), using the smoothness of $ B $, we get a morphism
$$ m_{f}^{\vee}\colon A^{\vee}\ra B^{\vee}\lten_{B^{e}}A^{e}. $$ Let $ \Xi $ be the cofiber of $ m_{f}^{\vee} $.
The dualizing bimodule $ \Theta_{f}=(\cof(B\lten_{B^{e}}A^{e}\ra A))^{\vee} $ of $ f $ is quasi-isomorphic to $ \Si^{-1}\Xi $. 

By the definition of Hochschild homology of $ f $, we have the following long exact sequence
\begin{equation}\label{equation:long exact sequence}
	\cdots\ra H\!H_{n-2}(B)\ra H\!H_{n-2}(A)\ra H\!H_{n-2}(f)\ra H\!H_{n-3}(B)\ra\cdots.
\end{equation}
Thus, the Hochschild homology class $ \xi\in H\!H_{n-2}(f) $ induces an element $ \xi_{B} $ in $ H\!H_{n-3}(B) $.

Notice that since $ B,A $ are smooth, we have the following isomorphisms$\colon$
\begin{equation*}
	\begin{split}
		\Hom_{\cd(B^{e})}(\Si^{n-2}B^{\vee},\Si B)\simeq& H^{3-n}(B\lten_{B^{e}}B)=HH_{n-3}(B),\\
		\Hom_{\cd(A^{e})}(\Si^{n-2}\Xi,\Si A)\simeq& H^{2-n}(\cone(B\lten_{B^{e}}A\ra A\lten_{A^{e}}A))\\
		\leftarrow&H^{2-n}(\cone(B\lten_{B^{e}}B\ra A\lten_{A^{e}}A))\\
		\simeq&HH_{n-2}(f).
	\end{split}
\end{equation*}

Thus the homology class $ [\xi] $ induces a morphism in $ \cd(A^{e}) $
$$ \xi\colon\Si^{n-2}\Xi\ra\Si A $$ and the homology class $ [\xi_{B}] $ induces a morphism in $ \cd(B^{e}) $
$$ \xi_{B}\colon\Si^{n-2}B^{\vee}\ra\Si B .$$
Moreover we have the following commutative diagram in $ \cd(A^{e}) $
\[
\begin{tikzcd}
	\bL f^{*}(\Si^{n-1}B^{\vee})\arrow[r]\arrow[d,"\xi_{B}"]&\Si^{n-2}\Xi\arrow[d,"\xi"]\\
	\bL f^{*}(\Si B)\arrow[r]&\Si A.
\end{tikzcd}
\]
Therefore, the morphism $ \xi_{B} $ gives rise to a ‘deformation’
$$ \bm\Pi_{n-1}(B,\xi_{B}) $$ of $ \bm\Pi_{n-1}(B)=\bL T_{B}(\Si^{n-2}B^{\vee}) $, obtained by adding $ \xi_{B} $ to the differential of $ \bm\Pi_{n-1}(B) $; the morphism $ \xi $ gives rise to a ‘deformation’
$$ \bm\Pi_{n}(A,B,\xi) $$ of $ \bm\Pi_{n}(A,B)=\bL T_{A}(\Si^{n-2}\Xi) $, obtained by adding $ \xi $ to the differential of $ \bL T_{A}(\Si^{n-2}\Xi) $; and the commutative diagram above gives rise to a morphism
$$ \tilde{f}\colon\bm\Pi_{n-1}(B,\xi_{B})\ra\bm\Pi_{n}(A,B,\xi) .$$

A standard argument shows that up to weak equivalence, the morphism $ \tilde{f} $ and the deformations $ \bm\Pi_{n-1}(B,\xi_{B}) $, $ \bm\Pi_{n}(A,B,\xi) $ only depend on the class $ \xi $.

\begin{Def}\rm\cite[Definition 3.14]{WkY2016}\label{Def:deformed CY completion}
	The dg functor $ \tilde{f} $ defined above is called the \emph{deformed relative n-Calabi--Yau completion} of $ f\colon B\ra A $ with respect to the Hochschild homology class $ \xi\in H\!H_{n-2}(f) $.
\end{Def}

\begin{Thm}\cite[Theorem 3.23]{WkY2016}\cite[Theorem 5.36]{BCS2020}\label{Thm:Relative CY completion to left steucture}\label{Thm:Relative defomed CY completion has a canonical left CY}
	If $ \xi $ has a negative cyclic lift, then each choice of such a lift gives rise to a canonical left $ n $-Calabi–Yau structure on the morphism
	$$ \tilde{f}\colon \bm\Pi_{n-1}(B,\xi_{B})\ra\bm\Pi_{n}(A,B,\xi) .$$
\end{Thm}

\section{Group actions on Calabi--Yau completions}\label{Section3}

\subsection{Group actions on dg morphisms}

Let $ G $ be a finite group. Denote by $ \id $ the neutral element of $ G $ and by $ k[G] $ the group algebra of $ G $. 
Let $ A $ be a dg algebra over $ k $.

\begin{Def}\rm
An \emph{action of $ G $ on $ A $ by dg automorphisms} is a morphism of complexes $\beta_{G}\colon k[G]\ten_{k} A\ra A $, where $ k[G] $ is a complex in degree $ 0 $ and with zero differential, such that, denoting by $ ^{g}a $ the image of $ g\ten a $, for all $ g\in G $ and $ a\in A  $, then 
	\begin{itemize}
		\item $^g(ab) = \,^ga\,^gb$ for all $ a,b\in A $ and $ g\in G $ and
		\item $ ^{\id}a=a $ and $ ^{gh}a=\,^{g}(^{h}a) $ for all $ a\in A $ and $ g,h\in G $. 
	\end{itemize}

Let $ f\colon B\ra A $ be a dg morphism (not necessarily unital). We say that \emph{$f$ is $G$-equivariant} if $G$ acts on both $B$ and $A$, and for each $b\in B$ and $g\in G$, we have $f(^{g}b)=\,^{g}\!f(b)$. 

\end{Def}

Now suppose that $ G $ acts on $ A $ by dg automorphisms. For a dg $A$-module $M$, a \emph{compatible action of $ G $ on $ M $ by dg automorphisms} is a morphism of complexes $k[G]\ten_{k} M\ra M $ such that, denoting by $ ^{g}m $ the image of $ g\ten m $, for all $ g\in G $ and $ m\in M  $, then 
\begin{itemize}
	\item $^g(m+n) = \,^gm+\,^gn$ for all $ m,n\in M $ and $ g\in G $ and
	\item $ ^{\id}m=m $ and $ ^{gh}m=\,^{g}(^{h}m) $ for all $ m\in M $ and $ g,h\in G $.
	\item $^{g}(m\cdot a)=\,^{g}m\cdot\,^{g}a $ for all $g\in G$, $m\in M$ and $a\in A$.
\end{itemize}

\begin{Def}\rm\label{Def: Skew-dg-algebra}
	The \emph{skew group dg algebra $A*G$} is the dg algebra defined as follows.
	\begin{itemize}
		\item The underlying complex of vector spaces equal to $ A\ten_{k}k[G] $, where any tensor $ a\ten g $ with $ a\in A $ and $ g\in G $ is denoted by $ a*g $;
		\item Multiplication is given by $$ (x* g)\cdot(y* h)=x\,^{g}\!y* gh $$ for all $ x,y\in A $ and $ g,h\in G $.
	\end{itemize}
\end{Def}
\begin{Rem}
	When $ A $ is a usual algebra, i.e. $ A=A^{0} $, the above definition coincides with the classical definition in~\cite{RR85}.
\end{Rem}

The action of $G$ on $A$ induces an action of $G$ on $ A^{e}$ which is defined as follows
$$
G\times A^{e}\ra A^{e} $$
$$\qquad\qquad\qquad\quad
(g,a\ten b)\mapsto\,^{g}(a\ten b)\coloneqq\,^{g}a\ten\,^{g}b.
$$
The corresponding skew group dg algebra is denote by $A^{e}*G$. Moreover, the action of $G$ on $A^{e}$ is compatible with the canonical right $A^{e}$-module structure on $A^{e}$.

\begin{Prop}\cite[Proposition 2.2 and Corollary 2.3]{Amiot-Plamondon2018}\label{Prop: cohomology of skew}
	The algebra $ A*G $ is a dg algebra. Moreover, $ G $ also acts on $ H^{0}(A) $ and $ H^{0}(A)*G=H^{0}(A*G) $ and the cohomology $H^{i}(A*G)$ is isomorphic to $H^{i}(A)\ten_{k}k[G]$ for each $i\in\mathbb{Z}$.
\end{Prop}

Denote by $\Lambda_{A}$ the skew group dg algebra $A*G$. We have a canonical morphism of dg algebras
$$\iota\colon A\ra\Lambda_{A}$$
$$\hspace{0.8cm} a\mapsto a*\id.$$


\subsection{The dg algebra $\triangle$ and its dg modules}
Let $\triangle_{A}$ be the following dg subalgebra of $\Lambda_{A}^{e}$
$$ \triangle_{A}\coloneqq\underset{g\in G}{\oplus}(A*g)\ten(A*g^{-1})\subseteq\Lambda_A^{e}=(A*G)\ten_{k}(A*G)^{op} .$$

Let $M$ be a dg $\triangle_{A}$-module. We can regard $M$ as a dg $A^{e}$-module with a compatible $G$-action as follows
\begin{itemize}
	\item for all $a,b\in A$, define $m\cdot(a\ten b)\coloneqq m\cdot((a*\id)\ten(b*\id))$,
	\item for all $g\in G$, define $^{g}m\coloneqq m\cdot((1*g^{-1})\ten(1*g))$.
\end{itemize} 
Conversely, let $N$ be a dg $A^{e}$-module with a compatible $G$-action. The structure of dg $\triangle_{A}$-module on $M$ is given by $m\cdot((a*g)\ten(a*g^{-1}))\coloneqq(\,^{g^{-1}}m)\cdot\,(^{g^{-1}}a\ten b)$. 

\begin{Lem}\cite[Section 3.1]{LeMeur2020}\label{Lemma:Tri_A=A^e-bimodule}
	The above shows that the category of $\triangle_{A}$-modules is equivalent to the category of dg $A^{e}$-modules with compatible $G$-actions.
\end{Lem}

For a dg $\triangle_{A}$-module $N$, we associate a dg $\Lambda_{A}^{e}$-module $M*G$ with underlying complex of vector spaces $M\ten_{k}k[G]$ and with action of $\Lambda_{A}^{e}$ such that $(m*g)\cdot(a*h\ten b*k)\coloneqq(\,^{k}m\cdot(^{kg}a\ten b))*kgh$ for all $m\in M$, $a,b\in A$ and $g,h,k\in G$.

Notice that $A$ can be endowed with the following dg $\triangle_{A}$-module structure
$$A\ten\triangle_{A}\ra A$$
$$\hspace{3.2cm}a\ten(b*g\ten c*g^{-1})\mapsto(-1)^{\deg(c)\cdot(\deg(b)+\deg(a))}c\,\,^{g^{-1}}\hspace{-0.3cm}a\,\,^{g^{-1}}\hspace{-0.07cm}b.$$
It is easy to see that the resulting dg $\Lambda_{A}^{e}$-module $A*G$ is $\Lambda_{A}$.
\begin{Prop}\cite[Lemma 3.1.1, Lemma 3.3.1 and Proposition 3.3.2]{LeMeur2020}\label{Prop:Props in LeMeur}
We have an isomorphism of dg algebras
$$ \triangle_{A}\ra A^{e}*G $$
$$ \hspace{-1.4cm} (a*g)\ten(b*g^{-1})\mapsto(a\ten\,^{g}b)*g .$$
Moreover, we have
\begin{itemize}
	\item [(1)] $\triangle_{A}\simeq(A^{e})^{\car(G)}$ as dg $A^{e}$-modules; $\Lambda_{A}^{e}\simeq\triangle_{A}^{\car(G)}$ as dg $\triangle_{A}$-modules;
	\item[(2)] Let $M\in\Mod(\triangle_{A})$, then
	\subitem(a) the following mapping is an isomorphism of dg $\Lambda_{A}^{e}$-modules
	$$M\ten_{\triangle_{A}}\Lambda_{A}^{e}\ra M*G$$
	$$\hspace{2.2cm}m\ten(a*g\ten b*h)\mapsto (m\cdot(a*h^{-1}\ten b*h))*hg.$$
	\subitem(b) $M$ is a direct summand of $M\ten_{\triangle_{A}}\Lambda_{A}^{e}$ in $\Mod(\triangle_{A})$. If $\charc(k)\nmid\car(G)$ then $M$ is a direct summand of $M\ten_{A^{e}}\triangle_{A}$ in $\Mod(\triangle_{A})$.
	\item[(3)] The restriction-of-scalars functor $\Mod(\triangle_{A})\ra\Mod(A^{e})$ maps cofibrant objects to cofibrant objects. Moreover, for all cofibrant resolution $X\ra A$ in in $\Mod(\triangle_{A})$, the composite morphism $$X*G\simeq X\ten_{\triangle_{A}}\Lambda_{A}^{e}\ra A\ten_{\triangle_{A}}\Lambda_{A}^{e}\ra A*G=\Lambda_{A}$$ is a cofibrant resolution in $\Mod(\Lambda_{A}^{e})$.
	
	\item[(4)] The functor $\Hom_{A^{e}}(-,A^{e})\colon \Mod(A^{e})\ra\Mod(A^{e})$ induces a functor $\Mod(\triangle_{A})\ra\Mod(\triangle_{A})$ also denoted by $\Hom_{A^{e}}(-,A^{e})$ and whose total derived functor $\RHom_{A^{e}}(-,A^{e})$ is such that $$\RHom_{A^{e}}(-,A^{e})\ten_{\triangle_{A}}\Lambda_{A}^{e}\simeq\RHom_{\Lambda_{A}^{e}}(-\ten_{\triangle_{A}}\Lambda_{A}^{e},\Lambda_{A}^{e}).$$
	In particular, we have $$A^{\vee}*G=\RHom_{A^{e}}(A,A^{e})\ten_{\triangle_{A}}\Lambda_{A}^{e}\simeq\RHom_{\Lambda_{A}^{e}}(A\ten_{\triangle_{A}}\Lambda_{A}^{e},\Lambda_{A}^{e})\simeq\RHom_{\Lambda_{A}^{e}}(\Lambda_{A},\Lambda_{A}^{e}).$$
\end{itemize}	
\end{Prop}

\subsection{Equivariant Cuntz-Quillen resolutions}

Let $\Omega^{1}(A)$ be the dg $A^{e}$-module $\ker(A\ten_{k}A\xrightarrow{m}A)$, where $m\colon A\ten_{k}A\ra A$ is the multiplication map. Then $\Omega^{1}(A)$ is generated by $D(f)\coloneqq f\ten1-1\ten f$ and it has a compatible $G$-action defined as $\,^{g}D(f)\coloneqq D(^{g}f)$. Hence there is a short exact sequence of dg $\triangle_{A}$-modules
$$0\ra\Omega^{1}(A)\xrightarrow{\alpha}A\ten_{k}A\xrightarrow{m}A\ra0.$$

Denote by $\cs(A)\in\Mod(\triangle_{A})$ the cone
$$\cs(A)\coloneqq\cone(\alpha\colon\Omega^{1}(A)\ra A\ten_{k}A)$$
called the \emph{Cuntz-Quillen resolution}. Clearly, there is a quasi-isomorphism in $\Mod(\triangle_{A})$
$$r_A\colon\cs(A)\twoheadrightarrow A $$
induced by the multiplication map. The underlying graded space of $\cs(A)$ is $\Si\Omega^{1}(A)\oplus (A\ten_{k}A)$. An element in $\Si\Omega^{1}(A)$ has the form $sD(f)$.

If $A$ is finitely cellular, i.e. there exists a finite graded quiver $Q$ such that $A=T_{kQ_0}(kQ_1)$, then $\cs(A)$ is cellular of finite rank, with basis $\{sD(f)\,|\,f\in Q_1\}\cup\{E_x\,|\,x\in Q_0\}$ where $E_x=e_x\ten e_x$ are the basis elements in $A\ten_k A$.

In particular, it is cofibrant and perfect (\cite[Corollary 2.19]{WkY2016}). By Proposition \ref{Prop:Props in LeMeur}, the canonical morphism $\cs(A)*G\ra A*G$ is a cofibrant replacement of $A*G$. 

Moreover, the $ A $-bimodule $ \cs(A)^{\vee} $ is also cellular of finite rank, with basis $ \{g^{\vee}|g\in Q_{1}\} \cup \{c_{y}|y\in Q_{0}\} $ where the arrow $ g^{\vee} $ has degree $ |g^{\vee}| =1-|g|$, and points in the opposite direction to $ g $; the loop $ c_{y} $ has degree $ |c_{y}|=0 $ , and is based at $ y $. In particular, it is cofibrant and perfect.

\subsection{Equivariant relative Calabi--Yau completions}
Let $G$ be a finite group. 
Let $ f\colon B\ra A $ be a dg morphism (not necessarily unital) between finitely cellular dg algebras such that $f$ is $G$-equivariant. By \cite[Remark 24.2.8]{WKY2016-thesis}, we can assume that $ f\colon B \to A$ is a semi-free extension, i.e.\ there is a finite graded quiver $ Q $ and a subquiver $ F\subseteq Q $ such that the underlying graded $ k $-category of $ B $ and $ A $ are isomorphic to $ T_{kF_{0}}(kF_{1}) $ and $ T_{kQ_{0}}(kQ_{1}) $, respectively. In particular, $B$ and $A$ are smooth (\cite[Corollary 2.22]{WkY2016}).

Let $r_{B}\colon \cs(B)\ra B$ (resp. $r_{A}\colon \cs(A)
\ra A$) be the Cuntz-Quillen resolution of $B$ (resp. $A$) in $\Mod(\triangle_{B})$ (resp. $\Mod(\triangle_{A})$). Then there exists a morphism $r_{f}\colon \cs(B)\ra \cs(A)$ of $\triangle_{B}$-modules such that the following square commutes in $\Mod(\triangle_{B })$
\begin{equation}\label{commutative diagram}
	\begin{tikzcd}
	\cs(B)\arrow[r,"r_{f}"]\arrow[d]&\cs(A)\arrow[d]\\
		B\arrow[r,"f"]&A.
	\end{tikzcd}
\end{equation}

Notice that $\cs(B)$ and $\cs(A)$ are also cofibrant in $\Mod(B^{e})$ and $\Mod(A^{e})$ respectively (see~\cite[Lemma 3.1.1]{LeMeur2020}).
By the above diagram~(\ref{commutative diagram}), we obtain a morphism of Hochschild chain complexes 
$$\Theta_{f}\colon \cs(B)\ten_{B^{e}}B\ra \cs(A)\ten_{A^{e}}A.$$
Moreover, it is $G$-equivariant provided that $\cs(B)\ten_{B^{e}}B$ and $ \cs(A)\ten_{A^{e}}A$ are endowed with the following actions of $G$
\begin{itemize}
	\item for all $g\in G$, $x\in \cs(B)$ and $b\in B$, let $^{g}(x\ten b)=\,^{g}x\ten\,^{g}b$;
	\item similar action for $\cs(A)\ten_{A^{e}}A$.
\end{itemize}

Then Hochschild homology $ H\!H_{\bullet}(f) $ of $ f $ is computed by $\cone(\Theta_{f})$ and $G$ also acts on $H\!H_{\bullet}(f)$. For each $k\in\mathbb{Z}$, denote by $\HH_{k}(B)^{G}$, $\HH_{k}(A)^{G}$, $\HH_{k}(f)^{G}$ the corresponding $G$-invariant subspace.

By the definition of $H\!H_{\bullet}(f)$, we have the following long exact sequence
\begin{equation*}\label{equation:long exact sequence1}
	\cdots\ra H\!H_{n-2}(B)\ra H\!H_{n-2}(A)\ra H\!H_{n-2}(f)\xrightarrow{\delta} H\!H_{n-3}(B)\ra\cdots.
\end{equation*}

And $\delta$ induces a map between the $G$-invariant subspaces
\begin{equation}\label{equation:G-map}
	\delta^{G}\colon H\!H_{n-2}(f)^{G}\ra H\!H_{n-3}(B)^{G}.
\end{equation}
 
\bigskip

The morphism $r_{f}\colon \cs(B)\ra \cs(A)$ of dg $\triangle_{B}$-modules in (\ref{commutative diagram}) induces a morphism $m_{f}\colon \cs(B)\ten_{\triangle_{B}}\triangle_{A}\ra \cs(A)$ of dg $\triangle_{A}$-modules. After taking the bimodule dual, we get a morphism in $\Mod(\triangle_{A})$
$$m_{f}^{\vee}\colon \cs(A)^{\vee}\ra \cs(B)^{\vee}\ten_{\triangle_{B}}\triangle_{A}.$$

Equivalently, we get a morphism of dg $A^{e}$-modules
$$s_{f}^{\vee}\colon\cs(A)^{\vee}
\ra\cs(B)^{\vee}\ten_{B^{e}}A^{e}$$
which is compatible with $G$-actions on $\cs(A)^{\vee}$ and $\cs(B)^{\vee}\ten_{B^{e}}A^{e}$ (see Lemma~\ref{Lemma:Tri_A=A^e-bimodule}).

Notice that $\cs(A)^{\vee}$ and $\cs(B)^{\vee}\ten_{B^{e}}A^{e}$ are also cofibrant. Let $ \Xi $ be the cone of $ s_{f}^{\vee} $. Then $\Xi$ is a dg $A^{e}$-module with a compatible $G$-action, i.e. $\Xi$ is a $\triangle_{A}$-module.

\begin{Rem}
	The $ B $-bimodule $ \cs(B)^{\vee} $ is cellular of finite rank, with basis $ \{f^{\vee}|f\in F_{1}\} \cup \{c_{x}|x\in F_{0}\} $ where the arrow $ f^{\vee} $ has degree $ |f^{\vee}| =1-|f|$, and points in the opposite direction to $ f $; the loop $ c_{x} $ has degree $ |c_{x}|=0 $ , and is based at $ x $. Similarly, the $ A $-bimodule $ \cs(A)^{\vee} $ is also cellular of finite rank, with basis $ \{g^{\vee}|g\in Q_{1}\} \cup \{c_{y}|y\in Q_{0}\} $ where the arrow $ g^{\vee} $ has degree $ |g^{\vee}| =1-|g|$, and points in the opposite direction to $ g $; the loop $ c_{y} $ has degree $ |c_{y}|=0 $ , and is based at $ y $.
	The map
	$$s_{f}^{\vee}\colon\cs(A)^{\vee}
	\ra\cs(B)^{\vee}\ten_{B^{e}}A^{e}$$ is surjective. Let $\ck$ be the kernel of $s_{f}^{\vee}$. Then $\Xi$ is quasi-isomorphic to $\Si\ck$ \cite[Section 3.6]{Wu2023}.
\end{Rem}

Let $\xi$ be an element of $\HH_{n-2}(f)^{G}\subseteq\HH_{n-2}(f)$ and let $ \xi_{B}=\delta^{G}(\xi)\in H\!H_{n-3}(B)^{G}\subseteq H\!H_{n-3}(B) $, where $\delta^{G}$ is the map (\ref{equation:G-map}).

The homology class $ \xi\in\HH_{n-2}(f)^{G}\subseteq\HH_{n-2}(f) $ induces a morphism in $ \cd(A^{e}) $
$$ \xi\colon\Si^{n-2}\Xi\ra\Si A $$ which is compatible with the $G$-actions, and the homology class $ \xi_{B} $ induces a morphism in $ \cd(B^{e}) $
$$ \xi_{B}\colon\Si^{n-2}\cs(B)^{\vee}\ra\Si B $$ which is also compatible with the $G$-actions.
Moreover, we have the following commutative diagram in $ \cd(A^{e}) $
\begin{equation}\label{equation:diagram S and B}
	\begin{tikzcd}
		\Si^{n-2}\cs(B)^{\vee}\ten_{B^{e}}A^{e}\arrow[r,"l"]\arrow[d,"\xi_{B}\ten\id"]&\Si^{n-2}\Xi\arrow[d,"\xi"]\\
		\Si B\ten_{B^{e}}A^{e}\arrow[r]&\Si A.
	\end{tikzcd}
\end{equation}

Therefore, the morphism $ \xi_{B} $ gives rise to a ‘deformation’
$$ \bm\Pi_{n-1}(B,\xi_{B}) $$ of $ \bm\Pi_{n-1}(B)=T_{B}(\Si^{n-2}\cs(B)^{\vee}) $, obtained by adding $ \xi_{B} $ to the differential of $ T_{B}(\Si^{n-2}\cs(B)^{\vee}) $; the morphism $ \xi $ gives rise to a ‘deformation’
$$ \bm\Pi_{n}(A,B,\xi) $$ of $ \bm\Pi_{n}(A,B)=T_{A}(\Si^{n-2}\Xi) $, obtained by adding $ \xi $ to the differential of $ T_{A}(\Si^{n-2}\Xi) $; and the commutative diagram above gives rise to a morphism
$$ \tilde{f}\colon\bm\Pi_{n-1}(B,\xi_{B})\ra\bm\Pi_{n}(A,B,\xi) .$$

Hence, by constructions, the dg algebras $\bm\Pi_{n-1}(B,\xi_{B})$ and $\bm\Pi_{n}(A,B,\xi)$ have natural $G$-actions. Moreover, the dg morphism $\tilde{f}$ is compatible with $G$-actions, i.e. $G$ acts on $\tilde{f}$. Therefore, we obtain a dg morphism between the corresponding skew group dg algebras
$$\tilde{f}*G\colon\bm\Pi_{n-1}(B,\xi_{B})*G\ra\bm\Pi_{n}(A,B,\xi)*G$$
given by
\begin{itemize}
	\item for each $g\in G$ and $b\in B$, we have $(\tilde{f}*G)(b*g)=f(b)*g$;
	\item for each $g\in G$ and $\alpha_1\ten\alpha_2\ten\cdots\ten\alpha_n$ in $(\Si^{n-1}\cs(B)^{\vee})^{\ten^{n}}$, we have $$(\tilde{f}*G)((\alpha_1\ten\alpha_2\ten\cdots\ten\alpha_n)*g)=(l(\alpha_1)\ten l(\alpha_2)\cdots\ten l(\alpha_n))*g $$ where $l$ is the map in (\ref{equation:diagram S and B}).
\end{itemize}

On the other hand, the action of $\alpha$ on $f$ also induces a dg morphism between the corresponding skew group dg algebras
$$f*G\colon B*G\ra A*G$$
given by
$$\hspace{1.8cm}b*h\mapsto f(b)*h. $$ 
If $\charc(k)\nmid\car(G)$, by \cite[Proposition 3.2.1]{LeMeur2020}, the skew group dg algebras $B*G$ and $A*G$ are smooth.

By the functoriality of the Hochschild homology, we have the following commutative diagram
\begin{equation*}\label{equation:2long exact sequences}
	\begin{tikzcd}
		\cdots\ra H\!H_{n-2}(B)\arrow[r]\arrow[d]& H\!H_{n-2}(A)\arrow[r]\arrow[d]& H\!H_{n-2}(f)\arrow[r,"\delta"]\arrow[d,"\phi"]& H\!H_{n-3}(B)\arrow[d,"\kappa"]\ra\cdots\\
		\cdots\ra H\!H_{n-2}(B*G)\arrow[r]& H\!H_{n-2}(A*G)\arrow[r]& H\!H_{n-2}(f*G)\arrow[r,"\omega"]& H\!H_{n-3}(B*G)\ra\cdots.
	\end{tikzcd}
\end{equation*}

\begin{Thm}\label{Thm:skew CY completion}
Let $G$ be a finite group. Let $ f\colon B\ra A $ be a dg morphism (not necessarily unital) between finitely cellular dg algebras such that $G$ acts on $f$. Suppose that $\charc(k)$ doesn't divide $\car(G)$. 
\begin{itemize}
	\item[(1)] The dg morphism
	$ \bm\Pi_{n}(B)*G\ra\bm\Pi_{n+1}(A,B)*G $ is equivalent to $$ \bm\Pi_{n}(B*G)\ra\bm\Pi_{n+1}(A*G),$$ i.e. the relative $ (n+1) $-Calabi--Yau completion of $ B*G\ra A*G $.
	\item[(2)] Let $\xi$ be an element of $\HH_{n-2}(f)^{G}$ and $ \xi_{B}=\delta^{G}(\xi)\in H\!H_{n-3}(B)^{G} $. Let $\varepsilon=\phi(\xi)$ and $\varepsilon_B=\omega(\varepsilon)$. Then the dg morphism
	$ \bm\Pi_{n}(B,\xi_B)*G\ra\bm\Pi_{n+1}(A,B,\xi)*G $ is equivalent to $ \bm\Pi_{n}(B*G,\varepsilon_B)\ra\bm\Pi_{n+1}(A*G,\varepsilon) $, i.e. the relative deformed $ (n+1) $-Calabi--Yau completion of $ B*G\ra A*G $ with respect to $\varepsilon\in H\!H_{n-2}(f*G)$. In particular, if $\xi$ lifts to $\HN_{n-2}(f)$, then $$\bm\Pi_{n}(B,\xi_{B})*G\ra\bm\Pi_{n+1}(A,B,\xi)*G$$ has a canonical left $(n+1)$-Calabi--Yau structure.
\end{itemize}

\end{Thm}\label{Thm:G-CY completion}
\begin{proof}
The proof is similar in spirit to \cite[Theorem 3.5.4]{LeMeur2020}. 

\end{proof}

%
%
%
%
%
%
%

\subsection{Application to Ginzburg dg functors}\label{Section 3.5}

Let $Q=(Q_{0},Q_{1})$ be a finite graded $k$-quiver.  Recall that the space of potentials on $Q$ is the graded vector space

$$\frac{kQ\ten_{kQ^{e}_{0}}kQ_{0}}{\langle uv\ten e_{x}-(-1)^{\deg(u)\deg(v)}vu\ten e_{y}\,|\,u\in e_{x}kQe_{y}, v\in e_{y}kQe_{x}\,\text{homogeneous, $x,y\in Q_{0}$}\rangle}.$$

Therefore, potentials are expressed as linear combinations of oriented cycles in Q, with each cycle being considered up to cyclic permutation with Koszul-type signs. Moreover, a homogeneous potential of degree $-d$ on $Q$ can be viewed as an element of $\HH_d(kQ)$. Let $F=(F_{0},F_{1})$ be a graded subquiver of $Q$ (which may not be full).  Let $n$ be a positive integer. Let $W$ be a homogeneous potential on $Q$ of degree $3-n$.

\begin{Def}\rm\label{Def:Relative Ginzburg algebra}
	Let $ \widetilde{Q} $ be the graded quiver with the same vertices as $ Q $ and whose arrows are
	\begin{itemize}
		\item the arrows of $ Q $,
		\item an arrow $ a^{\vee}:j\to i $ of degree $ 2-n-|a| $ for each unfrozen arrow $ a\colon i\ra j $,
		\item a loop $ t_{i}:i\to i $ of degree $ 1-n $ for each unfrozen vertex $ i $.
	\end{itemize}
	
	Define the \emph{$n$-dimensional relative Ginzburg dg algebra} $ \bm\Gamma_{n}(Q,F,W) $ as the dg algebra whose underlying graded space is the completed graded path algebra $k\widetilde{Q} $. Its differential is the unique $ k $-linear continuous endomorphism of degree 1 which satisfies the Leibniz rule
	\begin{align*}
		\xymatrix{
			d(u\circ v)=d(u)\circ v+(-1)^{p}u\circ d(v)
		}
	\end{align*}
	for all homogeneous $ u $ of degree $ p $ and all $ v $ and takes the following values on the arrows of $ \widetilde{Q} $:
	\begin{itemize}
		\item $ d(a)=0 $ for each arrow $ a $ of $ Q $,
		\item $d(a^{\vee})=\partial_{a}W$ for each unfrozen arrow $ a $,
		\item $ d(t_{i})=e_{i}(\sum_{a\in Q_{1}}[a,a^{\vee}])e_{i} $ for each unfrozen vertex $ i $, where $ e_{i} $ is the lazy path corresponding to the vertex $ i $ and  $[\,,\,]$ denotes the supercommutator.
	\end{itemize}
\end{Def}

\begin{Def}\rm\label{Def: derived preprojective algebra}
Let $ \widetilde{F} $ be the graded quiver with the same vertices as $ F $ and whose arrows are
	\begin{itemize}
		\item the arrows of $ F $,
		\item an arrow $ \tilde{a}:j\to i $ of degree $ 3-n-|a|$ for each arrow $ a $ of $ F $,
		\item a loop $ r_{i}:i\to i $ of degree $ 2-n $ for each vertex $ i $ of $ F $.
	\end{itemize}
	Define \emph{derived preprojective algebra} $ \bm\Pi_{n-1}(F) $ as the dg algebra whose underlying graded space is the completed graded path algebra $ k\widetilde{F} $. Its differential is the unique $ k $-linear continuous endomorphism of degree 1 which satisfies the Leibniz rule
	\begin{align*}
		\xymatrix{
			d(u\circ v)=d(u)\circ v+(-1)^{p}u\circ d(v)
		}
	\end{align*}
	for all homogeneous $ u $ of degree $ p $ and all $ v $, and takes the following values on the arrows of $ \widetilde{F} $:
	\begin{itemize}
		\item $ d(a)=0 $ for each arrow $ a $ of $ F $,
		\item $d(\tilde{a})=0 $ for each arrow $ a $ in $ F $,
		\item $ d(r_{i})=e_{i}(\sum_{a\in F_{1}}[a,\tilde{a}])e_{i} $ for each vertex $ i $ of $ F $, where $ e_{i} $ is the lazy path corresponding to the vertex $ i $.
	\end{itemize}
\end{Def}

Let $c$ be the image of $W$ under Connes’ boundary in Hochschild homology $\HH_{n-3} (kQ) \ra\HH_{n-2}(kQ)$ (see~\cite[Section 6.1]{BK2011}). Let $ \pi\colon kF\hookrightarrow kQ $  be the canonical dg inclusion. Then $ \xi=(0,c) $ is an element of $ H\!H_{n-2}(G) $.  Via the deformed relative $n$-Calabi--Yau completion of $ G $ with respect to the class $ \xi $,  we get a dg functor
$$ \bm \tilde{\pi}\colon\bm\Pi_{n-1}(F)\ra\bm\Pi_{n}(kQ,kF,\xi) ,$$ 
which has a canonical left $ n $-Calabi--Yau structure (see \cite[Theorem 3.23]{WkY2016}\cite[Theorem 5.36]{BCS2020}).

\begin{Prop}\cite[Proposition 3.18]{Wu2023}\label{Prop: reduce CY}
	We have the following commutative diagram (up to homotopy) in the category of pseudocompact dg algebras
	\[
	\begin{tikzcd}
		\bm\Pi_{n-1}(F)\arrow[r,"\tilde{\pi}"]\arrow[dr,"\bm G_{rel}",swap]&\bm\Pi_{n}(kQ,kF,\xi)\\
		&\bm{\Gamma}_{n}(Q,F,W)\arrow[u,hook,"i"]
	\end{tikzcd}
	\]
	Where $ i $ is a quasi-equivalent dg inclusion and $ \bm G_{rel} $ is given explicitly as follows:
	\begin{itemize}
		\item $ \bm G_{rel}(i)=i $ for each frozen vertex $ i\in F_{0} $,
		\item $ \bm G_{rel}(a)=a $ for each arrow $ a\in F_{1} $,
		\item $ \bm G_{rel}(\tilde{a})=-\partial_{a}W $ for each arrow $ a\in F_{1} $,
		\item $ \bm G_{rel}(r_{i})=e_{i}(\sum_{a\in Q_{1}\setminus F_{1}}[a,a^{\vee}])e_{i} $ for each frozen vertex $ i\in F_{0} $.
	\end{itemize}
	Thus, the functor $ \bm G_{rel}: \bm\Pi_{n-1}(F)\ra \bm{\Gamma}_{n}(Q,F,W) $ has a canonical left $ n $-Calabi--Yau structure.
\end{Prop}

We call $ \bm G_{rel}\colon\bm\Pi_{n-1}(F)\ra\bm{\Gamma}_{n}(Q,F,W) $ the \emph{$n$-dimensional Ginzburg functor} associated with $ (Q,F,W) $. 

\bigskip

Let $G$ be a finite group such that $\charc(k) \nmid\mathrm{Card}(G)$. We make the following assumption.
\begin{assumption}\label{ass1}\ 
	
	\begin{itemize}
		\item $G$ acts on $kQ$ by degree preserving automorphisms in such a way that both the set of vertices and the graded vector space generated by the arrows of $Q$ are stabilized by the action, and $F=(F_{0},F_{1})$ is also stable under the $G$-action. And the two $G$-actions are compatible under the inclusion $F\subseteq Q$.
		\item  $W$ is $G$-invariant up to cyclic permutation.
	\end{itemize}
\end{assumption}


\begin{Prop}\label{Prop: action on relative Ginzburg alg}\cite[Lemma 4.3.1]{LeMeur2020} Assume that $W$ is $G$-invariant up to cyclic permutation. Then the $n$-dimensional Ginzburg functor $ G_{rel}\colon\bm{\Pi}_{n-1}(kF)\ra\bm{\Gamma}_{n}(Q,F,W) $ is $G$-equivariant.  
\end{Prop}
\begin{proof}
	This follows from a similar discussion in \cite[Section 4.3]{LeMeur2020}.
\end{proof}



\begin{Def}\rm\cite{LD2010}\cite[Deﬁnition 4.4.1]{LeMeur2020}\label{Def: Q_{G} and F_{G}}
Let $[G\backslash Q_0]$ be a complete set of representatives of the $G$-orbits of vertices of $Q$. For each $i \in Q_0$, denote by $G_i$ the stabilizer of $i$ and let $[G/G_i]$ be a complete set of representatives of the cosets of $G$ modulo $G_i$. Finally, for all $i \in Q_0$, let $\text{irr}(G_i)$ be a complete set of representatives of the isomorphism classes of the irreducible representations of $G_i$; it is convenient to assume that $\rho = kG_i e_\rho$ for some primitive idempotent $e_\rho$ of $kG_i$, for all $\rho \in \text{irr}(G_i)$.
\begin{itemize}
	\item[(1)] Let $\epsilon$ be the following idempotent of $kQ*G$
	
	\[\epsilon = \sum_{i \in [G\backslash Q_0], \rho \in \text{irr}(G_i)} e_i * e_\rho\,,\]
	
	where $e_i$ ranges over a complete set of representatives of the $G$-orbits of vertices of $Q$ and $\rho$ ranges over a complete set of representatives of the isomorphism classes of the irreducible representations of $G_i$.
	\item[(2)] Let $Q_{G}$ be the quiver defined as follows
	\subitem$\star$ Its vertices are the pairs $(i, \rho)$ where $i \in [G\backslash Q_0]$ and $\rho \in \text{irr}(G_i)$.
	\subitem$\star$ For all vertices $(i, \rho)$ and $(j, \tau)$, denote by $M(i, j; \tau)$ the following vector subspace of $kQ*G$
	
	\[ M(i, j; \tau) = \bigoplus_{y\in[G/G_{j}]} (e_i (kQ_1) e_{y\cdot j}) * (y (kG_j) e_\tau). \]
	Note that $kQ*G$ is a representation of $G$ through left multiplication and, by restriction, it is a representation of $G_i$; then $M(i, j; \tau)$ is a sub-representation of this left $G_i$-module. 
	\subitem$\star$ The vector space spanned by the set of arrows in $Q_{G}$ from $(i, \rho)$ to $(j, \tau)$ is $\text{Hom}_{G_i}(\rho, M(i, j; \tau))$.
\end{itemize}

By \cite[Section 4.4]{LeMeur2020}, there exists an isomorphism of algebras
\[ \gamma_{Q}\colon kQ_{G} \iso \epsilon (kQ*G) \epsilon \]
given as follows
\begin{itemize}
	\item For all vertices $(i,\rho)$ of $Q_{G}$, the corresponding idempotent of $kQ_{G}$ is mapped onto $e_i * e_\rho$.
	\item For all vertices $(i, \rho)$ and $(j,\tau)$ of $Q_{G}$, and for all $f \in \text{Hom}_{G_i}(\rho, M(i, j; \tau))$, then $f$ is mapped onto $f(e_\rho)$. The arrow $f : (i, \rho) \rightarrow (j,\tau)$ of $Q_{G}$ can be identified with the corresponding element $f(e_\rho)$ of $(e_i * e_\rho) \cdot (kQ*G) \cdot (e_j * e_\tau)$.
\end{itemize}

\bigskip

Similarly, Let $\theta$ be the following idempotent of $kF*G$,

\[\theta = \sum_{i \in [G\backslash F_0], \rho \in \text{irr}(G_i)} e_i * e_\rho.\]

Let $F_{G}$ be the quiver defined as follows
\begin{itemize}
	\item[$\star$] Its vertices are the pairs $(i, \rho)$ where $i \in [G\backslash F_0]$ and $\rho \in \text{irr}(G_i)$.
	\item[$\star$] For all vertices $(i, \rho)$ and $(j, \tau)$, denote by $N(i, j; \tau)$ the following vector subspace of $kF*G$:
	
	\[ N(i, j; \tau) = \bigoplus_{y\in[G/G_{j}]} (e_i (kF_1) e_{y\cdot j})*(y (kG_j) e_\tau). \] 
	\item[$\star$] The vector space spanned by the set of arrows in $Q_{G}$ from $(i, \rho)$ to $(j, \tau)$ is $\text{Hom}_{G_i}(\rho, N(i, j; \tau))$.
\end{itemize}

\end{Def}

\bigskip

By \cite[Section 4.4]{LeMeur2020}, we also have an isomorphism of algebras
\begin{equation}\label{mor:gammaF}
	\gamma_{F}\colon kF_{G} \iso \theta (kF*G) \theta
\end{equation}

given as follows
\begin{itemize}
	\item For all vertices $(i,\rho)$ of $F_{G}$, the corresponding idempotent of $kF_{G}$ is mapped onto $e_i * e_\rho$.
	\item For all vertices $(i, \rho)$ and $(j,\tau)$ of $F_{G}$, and for all $f \in \text{Hom}_{G_i}(\rho, N(i, j; \tau))$, then $f$ is mapped onto $f(e_\rho)$. The arrow $f\colon(i, \rho) \rightarrow (j,\tau)$ of $F_{G}$ can be identified with the corresponding element $f(e_\rho)$ of $(e_i * e_\rho) \cdot (kQ*F) \cdot (e_j * e_\tau)$.
\end{itemize}

It is not hard to see that $F_{G}$ is a graded sub-quiver of $Q_{G}$. And the inclusion of algebras
$$\iota\colon kF_{G}\hookrightarrow kQ_{G}$$
induces the following commutative diagram
\begin{equation}\label{comm-digram}
	\begin{tikzcd}
		kF_{G}\arrow[r,hookrightarrow,"\iota"]\arrow[d,"\simeq"]&kQ_{G}\arrow[d,"\simeq"]\\
		\theta (kF*G)\arrow[d,hookrightarrow] \theta\arrow[r,hookrightarrow]&\epsilon (kQ*G) \epsilon\arrow[d,hookrightarrow]\\
		kF*G\arrow[r,hookrightarrow]&kQ*G.
	\end{tikzcd}
\end{equation}

\begin{Lem}\cite[Lemma 4.4.2]{LeMeur2020}\label{Lem:W_{G}}
	Under Assumption \ref{ass1}, there exists a homogeneous potential $W_G$ of degree $3-n$ on $Q_{G}$ such that the image of $W$ under the mapping $\HH_{n-3}(kQ) \to \HH_{n-3}(kQ*G)$ induced by the natural embedding $kQ \to kQ*G$ is equal to the image of $W_{G}$ under the isomorphism $\HH_{n-3}(kQ_{G}) \to \HH_{n-3}(kQ*G)$ induced by $kQ_{G} \to kQ*G$.
\end{Lem}

\begin{Rem}\cite[Lemma 4.4.2, Section 4.5]{LeMeur2020}
The computation of a potential $W_G$ such as above is made in two steps
	\begin{itemize}
		\item[(Step 1)] Express $W_G$ as an element of $\epsilon (kQ*G) \epsilon$. In fact, there exists a complete family $(\epsilon_{i})_{j\in J}$ of primitive pairwise orthogonal idempotents of $kQ*G$ and there exists a subset $I\subseteq J$ such that
		
		\begin{itemize}
			\item[$\bullet$] the image of the map $kQ_{G} \to kQ * G$ is equal to $\epsilon (kQ \ast G) \epsilon$, where $\epsilon$ denotes $\sum_{i \in I} \epsilon_i$, and
			\item[$\bullet$] for all $j \in J$, there exists a unique $\alpha(j) \in I$ such that $\epsilon_j \cdot (kQ \ast G) \cong \epsilon_{\alpha(j)} (kQ * G)$ as graded $kQ*G$-modules.
		\end{itemize}
		For all $j\in J$, there exist homogeneous $a_{j}, b_{j}\in kQ*G$ such that $\epsilon_{j}=a_{j}b_{j}$ and $\epsilon_{\alpha(j)}= b_{j} a_{j}$. Then 
		\begin{eqnarray*}
			\begin{split}
				W=&\id_{kQ*G}\cdot W\cdot\id_{kQ*G}\\
				=&\sum_{j \in J} \epsilon_j W \epsilon_j \\
				=&\sum_{j \in J} a_j b_j W a_j b_j a_j b_j \\
				=&\sum_{i \in I} \epsilon_i (\sum_{\substack{j\,\text{s.t.}\alpha(j) = i}}\pm b_j W a_j) \epsilon_i,
			\end{split}
		\end{eqnarray*}
	where the sign is $(-1)^{\deg(b_{j})\cdot(\deg(a_{j})+\deg(W))}$.		
		
		\item[(Step 2)] Express the result of the first step as a linear combination of paths in $Q_{G}$.
	\end{itemize}
\end{Rem}

\begin{Rem}
The ice quiver $(Q_{G},F_{G})$ itself can be computed for any finite group using the work of \cite{LD2010}. In the case where the group is of order $2$, the potential $W_{G}$ was computed in \cite{Amiot-Plamondon2018} where they used this to describe the cluster category of a triangulated surface with punctures. More recently, $W_G$ was computed in \cite{GP2019} for $G$ any cyclic group under some assumptions on the action. For any finite abelian group $G$, \cite{GPP2019} gave an explicit construction of the potential $W_G$ as a linear combination of cycles in $Q_G$. Finally, an algorithm to compute $W_{G}$ for any finite group was given in \cite{LeMeur2020b}.

\end{Rem}
%
%
%
%

\begin{Prop}\cite[Corollary 4.4.3]{LeMeur2020}\label{Prop: Lemeur-Morita}
	Under Assumption \ref{ass1}, we have
	\begin{itemize}
		\item[(1)] The actions of $G$ on $kQ$ and $kF$ extend to an action of $G$ on $\bmgamma_{n}(Q,F,W)$ by dg automorphisms.
		\item[(2)] The action of $G$ on $kF$ extends to an action of $G$ on $\bm{\Pi}_{n-1}(F)$ by dg automorphisms. And the Ginzburg functor $ \bm G_{rel}\colon\bm\Pi_{n-1}(F)\ra\bm{\Gamma}_{n}(Q,F,W) $ is $G$-equivariant.
		\item[(3)] The induced functor $\bm{\Pi}_{n-1}(F)*G\ra\bmgamma_{n}(Q,F,W)*G$ has a canonical left $n$-Calabi--Yau structure.
		
		\item[(4)] The morphisms $kQ_{G}\xrightarrow{\simeq}{\epsilon(kQ*G)\epsilon}\hookrightarrow kQ*G$ and $kF_{G}\xrightarrow{\simeq}{\theta(kF*G)\theta}\hookrightarrow kF*G$ extend to (non-unital) dg algebra homomorphisms
		
		\[\bm{\Pi}_{n-1}(F_{G}) \rightarrow \bm{\Pi}_{n-1}(F)*G \]
		and
		\[\bmgamma_{n}(Q_{G},F_{G},W_G) \rightarrow \bmgamma_{n}(Q,F,W)*G \]
		such that we have the following commutative diagram
		\[
		\begin{tikzcd}
			\bm{\Pi}_{n-1}(F_{G})\arrow[r]\arrow[d]&\bmgamma_{n}(Q_{G},F_{G},W_G)\arrow[d]\\
			\bm{\Pi}_{n-1}(F)*G\arrow[r]&\bmgamma_{n}(Q,F,W)*G
		\end{tikzcd}
		\]
		and
		the restriction-of-scalars functors induce equivalences of triangulated categories
		\[ \cd(\bmgamma_{n}(Q,F,W)*G)\simeq\cd(\bmgamma_{n}(Q_{G},F_{G},W_G)) \]
		and 
		\[ \cd(\bm{\Pi}_{n-1}(F)*G)\simeq\cd(\bm{\Pi}_{n-1}(F_{G})) .\]
	\end{itemize}
\end{Prop}
\begin{Rem}
	The equivalence $\cd(\bmgamma_{n}(Q_{G},F_{G},W_G))\ra\cd(\bmgamma_{n}(Q,F,W)*G)$ maps $\bmgamma_{n}(Q_{G},F_{G},W_G)$ to $\epsilon(\bmgamma_{n}(Q,F,W)*G)$.
\end{Rem}

\begin{Cor}
	 The equivalence $ \cd(\bmgamma_{n}(Q_{G},F_{G},W_G))\iso\cd(\bmgamma_{n}(Q,F,W)*G)$ restricts to equivalences 
	 $$\per(\bmgamma_{n}(Q_{G},F_{G},W_G))\iso\per(\bmgamma_{n}(Q,F,W)*G)$$ and $$\pvd(\bmgamma_{n}(Q_{G},F_{G},W_G))\iso\pvd(\bmgamma_{n}(Q,F,W)*G).$$
\end{Cor}

\begin{Ex}
	Let $ (Q,F,W) $ be the following ice quiver with potential
	\[
	\begin{tikzcd}
		&\color{blue}\boxed{2}\arrow[dr,"e"]&&5\arrow[ll,"s",swap]&\\
		\color{blue}\boxed{1}\arrow[ur,blue,"a"]\arrow[dr,blue,"b",swap]&&4\arrow[dr,swap,"g"]\arrow[ur,"f"]\arrow[ll,"c"]&\\
		&\color{blue}\boxed{3}\arrow[ur,"d",swap]&&6\arrow[ll,"r"]&W=eac-fes+gdr-bcd.
	\end{tikzcd}
	\]
	
	By \cite[Corollary 8.7]{Wu2023}, the relative Ginzburg algebra $\bmgamma_{3}(Q,F,W)$ is quasi-isomorphic to $H^{0}(\bmgamma_{3}(Q,F,W))$. We define a $ \mathbb{Z}/2\mathbb{Z}=\langle\sigma\rangle $-action on $ (Q,F) $ as follows
	\begin{itemize}
		\item $\sigma(1)=1$, $\sigma(2)=3$, $\sigma(4)=4$ and $\sigma(5)=6$.
		\item $\sigma(a)=-b$, $\sigma(c)=c$, $\sigma(e)=d$, $\sigma(s)=r$ and $\sigma(f)=-g$.
	\end{itemize}
It is clear that $W$ is invariant up to cyclic permutation.
We compute the following data to define $Q_{G} $.
\begin{itemize}
	\item $[G\setminus Q_{0}]=\{1,2,4,5\}$, $G_{1}=G=G_{4}$ and $G_{2}=\{\id\}=G_{5}$
	\item $[G/G_{1}]=\{\id\}=[G/G_{4}]$ and $[G/G_{2}]=G=[G/G_{5}]$.
	\item The irreducible representations of \(G_1\) and \(G_4\) consist of the trivial representation \(\rho_+\) given by \(kG \cdot (\id +\sigma)\) and the non-trivial one \(\rho_-\) given by \(kG \cdot (\id-\sigma)\).
	\item The irreducible representations of \(G_2\) and $G_{5}$ consist of the trivial representation \(\mathbb{K}\) of $\{\id\}$.
\end{itemize}
Then we have $\epsilon=\epsilon_{1}^{+}+\epsilon_{1}^{-}+\epsilon_{4}^{+}+\epsilon_{4}^{-}+\epsilon_{2}+\epsilon_{5}$, where where \(\epsilon^{+}_{1} = \frac{1}{2} (e_1 \ast \id + e_1 \ast \sigma)=\epsilon^{+}_{4}\), \(\epsilon^{-}_{1} = \frac{1}{2} (e_1 \ast \id - e_1 \ast \sigma)=\epsilon^{-}_{4}\), and \(\epsilon_2 = e_2 \ast \id=\epsilon_5\). Hence
$$M(1,2;k)=\spa(a*\id,b*\sigma)\simeq\spa(\epsilon_{1}^{+}a\epsilon_{2})\oplus\spa(\epsilon_{1}^{-}a\epsilon_{2})\simeq\rho_+\oplus\rho_-,$$
$$M(4,5;k)=\spa(f*\id,g*\sigma)\simeq\rho_+\oplus\rho_-,$$
$$M(4,1;\rho_+)=\spa(c\epsilon^{+}_{1})\simeq\rho_+,$$
$$M(4,1;\rho_-)=\spa(c\epsilon^{-}_{1})\simeq\rho_-,$$
$$M(2,4;\rho_+)=\spa(e\epsilon^{+}_{1})\simeq\mathbb{K},$$
$$M(2,4;\rho_-)=\spa(e\epsilon^{+}_{1})\simeq\mathbb{K},$$
$$M(5,2;\mathbb{K})=\spa(s*\id)\simeq\mathbb{K}. $$

Then the quiver $Q_{G}$ is given as follows
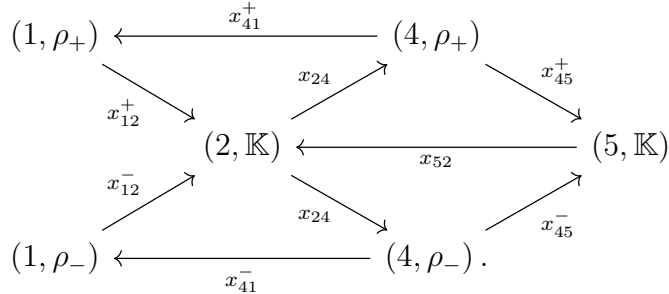
\[
\begin{tikzcd}
	(1,\rho_+)\arrow[dr,swap,"x_{12}^{+}"]&&(4,\rho_+)\arrow[ll,swap,"x_{41}^{+}"]\arrow[dr,"x_{45}^{+}"]&\\
	&(2,\mathbb{K})\arrow[ur,"x_{24}"]\arrow[dr,swap,"x_{24}"]&&(5,\mathbb{K})\arrow[ll,"x_{52}"]\\
	(1,\rho_-)\arrow[ur,"x_{12}^{-}"]&&(4,\rho_-)\arrow[ll,"x_{41}^{-}"]\arrow[ur,swap,"x_{45}^{-}"]\,.
\end{tikzcd}
\]

The potential $W_G$ on $Q_{G}$ is given by $x_{12}^{+}x_{41}^{+}x_{24}-x_{52}x_{45}^{+}x_{24}+x_{45}^{-}x_{24}x_{52}-x_{24}x_{12}^{-}x_{41}^{-}$.

Similarly, the quiver $F_{G}$ is given by the following full sub-quiver of $Q_{G}$
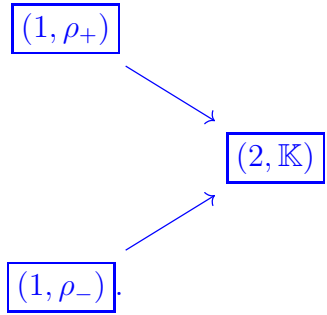
\[
\begin{tikzcd}
	\color{blue}\boxed{(1,\rho_+)}\arrow[dr,blue]&&\\
	&\color{blue}\boxed{(2,\mathbb{K})}\\
	\color{blue}\boxed{(1,\rho_-)}\arrow[ur,blue].
\end{tikzcd}
\]

We thus obtain an ice quiver with potential $(Q_{G},F_{G},W_G)$ and $\bmgamma_{3}(Q,F,W)*G$ is derived Morita equivalent to $ \bmgamma_{3}(Q_{G},F_{G},W_G)$.

\end{Ex}
Let $\overline{Q}$ be the quiver obtained from $Q$ by deleting all vertices in $F$ and all arrows incident with vertices in $F$. Let $\overline{W}$ be the potential on $Q$ obtaining by deleting all cycles passing through vertices of $F$ in $W$. We define $(\overline{Q_{G}},\overline{W_{G}})$ similarly.
\begin{Cor}\cite[Corollary 4.4.3]{LeMeur2020}\label{Cor: stable derived Morita}
	There is a non-unital dg algebra homomorphism
	$$\bmgamma_{n}(\overline{Q_{G}},\overline{W_G})\ra\bmgamma_{n}(\overline{Q},\overline{W})*G$$
	and
	the restriction-of-scalars functor induces the following equivalences of triangulated categories
	\[ \cd(\bmgamma_{n}(\overline{Q},\overline{W})*G)\simeq\cd(\bmgamma_{n}(\overline{Q_{G}},\overline{W_G})),\]
	\[ \per(\bmgamma_{n}(\overline{Q},\overline{W})*G)\simeq\per(\bmgamma_{n}(\overline{Q_{G}},\overline{W_G})),\]
	\[ \pvd(\bmgamma_{n}(\overline{Q},\overline{W})*G)\simeq\pvd(\bmgamma_{n}(\overline{Q_{G}},\overline{W_G})).\]
\end{Cor}

\begin{Rem}
Let $\mathbf{S_{w}}$ be a weighted marked surface \cite[Definition 2.1]{CHQ2024} and $\mathbb{A}$ a mixed-angulation of $\mathbf{S_{w}}$ \cite[Definition 2.3]{CHQ2024}. There is a branching covering $p\colon \widetilde{\mathbf{S_{w}}} \to \mathbf{S_{w}}$ such that $\mathbb{A}$ lifts to an $n$-angulation $\widetilde{\mathbb{A}}$ of $\widetilde{\mathbf{S_{w}}}$, where the quotient group $H$ is finite \cite[Proposition 3.25]{CHQ2024}. Let $\mathbb{S}$ and $\tilde{\mathbb{S}}$ be the corresponding $S$-graph of $(\mathbf{S_{w}},\mathbb{A})$ and $(\widetilde{\mathbf{S_{w}}},\widetilde{\mathbb{A})}$ respectively \cite[Definition 2.5]{CHQ2024}. Denote by $A(\mathbb{S},n)$ the associated relative graded Brauer graph algebra (RGB algebra) of $(\mathbb{S},n)$ \cite[Definition 2.8]{CHQ2024}. Let $G(\mathbb{S},n)$ be the dg algebra constructed in \cite[Construction 3.17]{CHQ2024} which is isomorphic to the Koszul dual of $A(\mathbb{S},n)$. 

By \cite[Proposition 3.25]{CHQ2024}, the dg algebra \( G(\mathbb{S},n) \) is obtained as the quotient of the relative Ginzburg dg algebra \( G(\widetilde{\mathbb{S}},n) \) associated with \( \widetilde{\mathbb{A}} \) by the action of \( H \). Hence \( G(\mathbb{S},n) \) is derived Morita equivalent to the skew-group dg algebra \( G(\widetilde{\mathbb{S}},n) * H \).

\end{Rem}

\section{Equivariant categories}\label{Section4:Equivariant categories}

Let $ A $ be a smooth dg $ k $-algebra and $ G $ a finite group acting on $ A $ by dg algebra automorphisms. We assume that $\charc(k)\nmid\car(G)$. Let $ \cc(A) $ be the category of right dg $ A $-modules. 
For each $g\in G$ denote the restriction
of scalars along the dg automorphism $A\ra A, a\mapsto\,^{g}a$ as follows
$$ ^{g}(\,)\colon\cc(A)\ra\cc(A)$$ $$\hspace{1cm}M\mapsto \,^{g}M $$
$$\hspace{1.1cm}f\mapsto \,^{g}f .$$
This defines a strict action of $G$ on $\cc(A)$ on the right by automorphisms of abelian categories. 

Recall that a \emph{G-equivariant object} \cite[Subsection 4.2]{Chen2015} in $ \cc(A) $ is a pair $ (M,\alpha) $, where $ M $ is an object in $ \cc(A) $ and $ \alpha $ assigns for each $ g\in G $ an isomorphism $ \alpha_{g}\colon \,^{g}M\ra M $ subject to the relations $ \alpha_{g}\circ\,^{g}(\alpha_{g'})=\alpha_{gg'} $. We denote by $ \cc(A)^{G} $ the category of $ G $-equivariant objects in $\cc(A)$. 


\begin{Lem}\rm\cite[Proposition 2.48]{LD2011}
	There is an equivalence of abelian categories $$ \cc(A)^{G}\simeq\cc(A*G) .$$ 
\end{Lem}
%
%
%

\begin{Rem}
	The image of $A*G$ under the equivalence $\cc(A*G)\ra\cc(A)^{G}$ is isomorphic to $(A[G],\mathrm{Id})$, where $A[G]$ is the direct sum $\oplus_{h\in G}\,^{h}A$ as a right $A$-module and for each $g\in G$, $\mathrm{Id}_{g}$ is the unique structural isomorphism from $^{g}(\oplus_{h\in G}\,^{h}A)$ to $\oplus_{h\in G}\,^{h}A$.
\end{Rem}

Let $ \cc $ be a category. Recall from \cite[Chapter VI]{MacLane1998} that a \emph{monad} on $ \cc $ is a triple $ (M, \eta, \mu) $ consisting of an endofunctor $ M\colon\cc\ra\cc $ and two natural transformations, the unit $ \eta\colon\id_{\cc}\ra M $ and the multiplication $ \mu\colon M\circ M\ra M $, subject to the relations $ \mu\circ M\mu=\mu\circ\mu M $ and $ \mu\circ M\eta=\id_{M}=\mu\circ\eta M $. We sometimes denote the monad by $ M $ when $ \eta $ and $ \mu $ are understood.

For a monad $ (M,\eta,\mu) $ on $ \cc $, an \emph{$ M $-module} is a pair $ (X, \lambda) $ consisting of an object $ X $ in $ \cc $ and a morphism $ \lambda\colon M(X)\ra X $ subject to the conditions $ \lambda\circ M\lambda=\lambda\circ\mu_{X} $ and $ \lambda\circ\eta_{X}=\id_{X} $; the object $ X $ is said to be the underlying object of the module. A morphism $ f\colon(X,\lambda)\ra (X',\lambda') $ of two $ M $-modules is a morphism $ f\colon X\ra X' $ in $ \cc $ satisfying $ f\circ\lambda=\lambda'\circ M(f) $. This defines the category $ M $-$ \Mod_{\cc} $ of $ M $-modules.

For each object $ X $ in $ \cc $, we have the corresponding $ M $-module $ (M(X),\mu_{X}) $, the \emph{free module}. This gives rise to a functor $ F_{M}\colon\cc\ra M$-$\Mod_{\cc} $ sending $ X $ to the free module $ (M(X),\mu_{X}) $, and a morphism $ f\colon X\ra Y $ to the morphism $ M(f)\colon (M(X),\mu_{X})\ra(M(Y),\mu_{Y}) $. Let $ G_{M}\colon M\text{-}\Mod_{\cc}\ra\cc $ be the forgetful functor. Then we have an adjoint pair $ (F_{M}, G_{M};\eta_{M},\epsilon_{M}) $. This adjoint pair $ (F_{M},G_{M};\eta_{M},\epsilon_{M}) $ defines the given monad $ M $, i.e. $ M=G_{M}\circ F_{M} $. Moreover, it satisfies the following universal property.

For any adjoint pair $ (F,G;\eta,\epsilon) $ on $ \cc $ and $ \cd $ that defines $ M $, there is a unique functor $ J\colon\cd\ra M\text{-}\Mod_{\cc} $ such that $ J\circ F=F_{M} $ and $ G_{M}\circ J=G $; see \cite[VI.3]{MacLane1998}. This unique functor $ J $ is given by $ J(D)=(G(D),G\epsilon_{D}) $ for any object $ D $ and $ J(f)=G(f) $ for any morphism $ f $, and we call $ J $ the \emph{comparison functor} associated to the adjoint pair $ (F,G;\eta,\epsilon) $.

In our situation, let $U$ be the forgetful functor 
$$U\colon\cc(A)^{G}\ra\cc(A)$$
$$\hspace{0.1cm}(X,\alpha)\mapsto X .$$

It admits a left adjoint
$$K\colon\cc(A)\ra\cc(A)^{G}$$
$$\hspace{2.4cm}X\mapsto(\oplus_{h\in G}\,^{h}X,\mathrm{Id}),$$
where for each $g\in G$,  $\mathrm{Id}_{g}\colon^{g}(\oplus_{h\in G}\,^{h}X)\ra\,\oplus_{h\in G}\,^{h}X$ is the unique structural isomorphism. Denote by $(M=UK,\eta,\mu=U\epsilon F)$ the monad on $\cc(A)$ defined by the adjoint pair $(K, U)$, where $\eta\colon\id_{\cc(A)}\ra UK$ is the unit and $\epsilon\colon KU\ra\id_{\cc(A)}$ is the counit. By \cite[Lemma 4.3]{Chen2015}, we can identify $\cc(A)^{G}$ with $M\text{-}\Mod_{\cc(A)}$. Since $K$ and $U$ are both exact, they extend to triangle functors \(\tilde{K}\colon \cd(A)\rightarrow\cc(A)^{G}[\mathrm{quasi}^{-1}]\) and \(\tilde{U}\colon\cc(A)^{G}[\mathrm{quasi}^{-1}]\rightarrow\cd(A)\). They still form an adjoint pair and the associated monad is denoted by $\tilde{M}$. Therefore we have the following comparison functor associated to the adjoint pair $(\tilde{K},\tilde{U})$
\begin{equation}\label{fun:K-}
	\tilde{K}\colon\cc(A)^{G}[\mathrm{quasi}^{-1}]\ra\tilde{M}\text{-}\Mod_{\cd(A)}.
\end{equation}

The $G$-action on $\cc(A)$ preserves
both projective and injective model structures of $\cc(A)$, and hence defines a strict action of $G$ on $\cd(A)$
by strict automorphisms of triangulated category. Then we have the forgetful functor \(U'\colon\cd(A)^G \rightarrow\cd(A)\) and its left adjoint \(K'\colon\cd(A) \rightarrow \cd(\cc(A))^G\). Observe that the associated monad $M'$ coincides with $\tilde{M}$. Therefore $\cd(A)^{G}$ can be identified with $\tilde{M}\text{-}\Mod_{\cd(A)}$ by \cite[Lemma 4.3]{Chen2015} and the comparison functor~(\ref{fun:K-}) is equivalent to
\begin{equation*}
	\tilde{K}\colon\cd(A*G)\simeq\cc(A)^{G}[\mathrm{quasi}^{-1}]\ra\cd(A)^{G}.
\end{equation*}

It is clear that the strict action of $G$ on $\cd(A)$ induces a strict action of $G$ on $\per(A)$ and $\pvd(A)$ respectively. Then $\tilde{K}$ induces the following commutative diagram
\begin{equation}\label{diagram:pvd and per}
	\begin{tikzcd}
		\pvd(A*G)\arrow[r,"\tilde{K}"]\arrow[d,hook]&\pvd(A)^{G}\arrow[d,hook]\\
		\per(A*G)\arrow[r,"\tilde{K}"]&\per(A)^{G}.
	\end{tikzcd}
\end{equation}
\begin{Prop}\label{Prop:equ of cats}
 We have the following equivalences of triangulated categories
$$ \tilde{K}\colon\per(A*G)\iso\per(A)^{G},\quad\pvd(A*G)\iso\pvd(A)^{G}.$$
Moreover, the image $\tilde{K}(A*G)$ is isomorphic to $(\oplus_{h\in G}\,^{h}A,\mathrm{Id})$.
\end{Prop}
\begin{proof}
Let us first show the equivalence $\tilde{K}\colon\per(A*G)\iso\per(A)^{G}$. Since $\per(A*G)$ is a triangulated category which is idempotent complete, by \cite[Lemma 4.4]{Chen2015}, the category $\per(A)^{G}$ has a unique pre-triangulated structure (a triangulated structure possibly without the octahedral axiom) such that the forgetful functor $U'\colon\per(A)^{G}\ra\per(A)$ is a triangle functor. It follows that functor $\tilde{K}\colon\per(A*G)\ra\per(A)^{G}$ is a triangle functor. 

By our assumption $\charc(k)\nmid\car(G)$ and \cite[Lemma 4.4(2)]{Chen2015}, the monad $\tilde{M}$ on $\per(A*G)$ is separable. Then the result follows from \cite[Proposition 4.1]{Chen2015}. The proof of the equivalence $\tilde{K}\colon\pvd(A*G)\ra\pvd(A)^{G}$ is similar.
\end{proof}

\bigskip

Let $ e $ be an idempotent of $ A $. Suppose that it admits a finite primitive orthogonal idempotents decomposition $ e=\sum_{i\in I}e_{i} $ and the set $ \{e_{i}|\,i\in I\} $ is stable under the action of $ G $. Hence the $G$-action on $A$ induces a $G$-action on the dg subalgebra $eAe$.  

Recall that Drinfeld’s construction of the homotopy cofiber of $i\colon eAe\hookrightarrow A$ is given as follows \cite[Section 7.2]{Kalck-Yang2018}:

Define a dg \( k \)-category \( \ca \) with two objects \( \epsilon \) and \( \gamma \) such that the morphism spaces are given by:
\begin{align*}
	\text{Hom}_\ca(\epsilon, \epsilon) &= eAe, \\
	\text{Hom}_\ca(\epsilon, \gamma) &= (1 - e)Ae, \\
	\text{Hom}_\ca(\gamma, \epsilon) &= eA(1 - e), \\
	\text{Hom}_\ca(\gamma, \gamma) &= (1 - e)A(1 - e),
\end{align*}
and the composition of morphisms is induced from the multiplication of $A$. Let \( \ca' \) be the dg subcategory of \( \ca \) consisting of the object \( \epsilon \).  We form the dg quotient \( \ca / \ca' \) by formally adjoining a morphism \( x: \epsilon \to \epsilon \) of degree \(-1\) such that \( d(x) = \text{id}_{\epsilon} = e \).
Then the dg algebra \( \overline{A} = \text{End}_{\ca / \ca'}(\gamma) \) is a homotopy cofiber of $i\colon eAe\hookrightarrow A$.
As a total complex, \( \overline{A} \) has the form
\begin{align*}
	\cdots & \longrightarrow Ae \otimes (eAe)^{\otimes p} \otimes eA \overset{d^{-p-1}}{\longrightarrow} \cdots \longrightarrow Ae \otimes eAe \otimes eA \overset{d^{-2}}{\longrightarrow} Ae \otimes eA \overset{d^{-1}}{\longrightarrow} A,
\end{align*}
where the rightmost term is in degree 0 and the term \( Ae \otimes (eAe)^{\otimes p} \otimes eA \) is in degree \(-p - 1\).
The differentials are given by
\[
d^{-p-1}(a_0 \otimes a_1 \otimes \cdots \otimes a_p \otimes a_{p+1}) = \sum_{i=0}^{p+1} (-1)^i a_0 \otimes \cdots \otimes a_i a_{i+1} \otimes \cdots \otimes a_{p+1}.
\]
The multiplication of \( \overline{A} \) is induced from that of \( A \)
\[
(a_0 \otimes \cdots \otimes a_{p+1})(b_0 \otimes \cdots \otimes b_{q+1}) = a_0 \otimes \cdots \otimes a_p \otimes a_{p+1} b_0 \otimes b_1 \otimes \cdots \otimes b_{q+1}.
\]

For each $p\geqslant0$, define a $G$-action on $Ae \otimes (eAe)^{\otimes p} \otimes eA$ by the following formula
$$g.(a_0 \otimes a_1 \otimes \cdots \otimes a_p \otimes a_{p+1})\coloneqq\,^{g}\!a_0 \otimes\,^{g}\!a_1 \otimes \cdots \otimes\,^{g}\!a_p \otimes\,^{g}\! a_{p+1}.$$
It is easy to see that this defines a $G$-action on $\overline{A}$ and the canonical morphism $A\ra\overline{A}$ is $G$-equivariant. Moreover, if $A$ is a dg quiver $(kQ,d)$, i.e. $Q$ is a graded quiver with finitely many vertices and the differential $d$ takes all trivial paths to $0$, then $\overline{A}$ is quasi-equivalent to $A/AeA$ \cite[Theorem 7.1]{Kalck-Yang2018}.
\bigskip

Let  $ \mathrm{pvd}_{e}(A) $ be the full triangulated subcategory of $ \mathrm{pvd}(A) $ defined as the kernel of the restriction functor $ i_{*}\colon\cd(A)\to\cd(eAe) $. By the recollement in \cite[Corollary 4.6]{Wu2023}, $ \mathrm{pvd}_{e}(A) $ is triangle equivalent to $\pvd(\overline{A})$. Hence $\pvd_{e}(A)$ admits a $G$-action induced from the $G$-action on $\cd(A)$.



\begin{Cor}
The triangle equivalence $\pvd(A)^{G}\simeq\pvd(A*G)$ restricts to the following equivalence of triangulated categories
	$$(\pvd_{e}(A))^{G}\simeq\pvd(\overline{A}*G).$$
\end{Cor}
\begin{proof}
	It follows from the equivalence $\pvd_{e}(A)\simeq\pvd(\overline{A})$ and Proposition \ref{Prop:equ of cats}.
\end{proof}

%

\section{Equivariant relative cluster categories and Equivariant Higgs categories}\label{Section5}
Let $ \ct $ be any triangulated category. Let $ \ct' $ be a full subcategory of $ \ct $. We denote by $ \pr_{\ct}\ct' $ the full subcategory of $ \ct $ whose objects are cones of morphisms in $ \add\ct' $. Similarly, we denote by $ \copr_{\ct}\ct' $ the full subcategory of $ \ct $ whose objects are those $ X $ such that $ \Si X $ is in $ \pr_{\ct}\ct' $. If $ \ct'=\add T $ for some object $ T\in\ct $, the categories $ \pr_{\ct}\ct' $ and $ \copr_{\ct}\ct' $ will be simply denoted by $ \pr_{\ct}T $ and $ \copr_{\ct}T $ respectively.

Let $ (Q,F,W) $ be an ice quiver with potential and $ G $ a finite group acting on $ (Q,F,W) $, i.e. the Assumption~\ref{ass1} is satisfied. Denote by $ \bm{\Gamma} $ the corresponding relative Ginzburg algebra, i.e. the 3-dimensional relative Ginzburg dg algebra. Let $ e=\sum_{i\in F}e_{i} $ be the idempotent associated with the set of frozen vertices and $ \cp $ the additive subcategory $ \add(e\bmgamma) $ of $ \per\bmgamma $. The action of $ G $ on $ (Q,F,W) $ induces an action of $ G $ on $ \bm{\Gamma} $ by dg algebra automorphisms, and the dg subalgebra $e\bmgamma e$ is invariant under this $G$-action.

Let $\overline{Q}$ be the quiver obtained from $Q$ by deleting all vertices in $F$ and all arrows incident with vertices in $F$. Let $\overline{W}$ be the potential on $Q$ obtaining by deleting all cycles passing through vertices of $F$ in $W$. We make the following assumption.
\begin{assumption}\label{ass2}\ 
	
	\begin{itemize}
		\item The quiver with potential $(\overline{Q},\overline{W})$ is Jacobi-finite, i.e. the corresponding Jacobian algebra $H^{0}(\bmgamma(\overline{Q},\overline{W}))$ is finite dimensional;
		\item The additive subcategory $\cp=\add(e\bmgamma)$ is functorially finite in $\add(\bmgamma)$.
	\end{itemize}
\end{assumption}
In particular, if $(Q,F,W)$ is relative Jacobi-finite, i.e. $H^{0}(\bmgamma(Q,F,W))$ is finite dimensional, then the assumption above holds.

\begin{Lem}
	If $(Q,F,W)$ satisfies Assumption \ref{ass2}, then the ice quiver with potential $(Q_{G},F_{G},W_G)$ constructed in Definition \ref{Def: Q_{G} and F_{G}} also satisfies the corresponding assumption.
\end{Lem}
\begin{proof}
	This follows from Proposition \ref{Prop: cohomology of skew}, \ref{Prop: Lemeur-Morita} and Corollary \ref{Cor: stable derived Morita}.

\end{proof}

\begin{Def}\rm\cite[Definition 3.1]{KellerWu2023}
	The \emph{relative cluster category} $ \cc(Q,F,W) $ (or denoted by $ \cc(\bmgamma) $) of $ (Q,F,W) $ is defined as the idempotent completion of the Verdier quotient of triangulated categories $$ \mathrm{per}(\bm\Gamma)/\mathrm{pvd}_{e}(\bm\Gamma) .$$
	If $ F=\emptyset $, the \emph{cluster category} associated with $ (Q,W) $ is defined as $ \cc(Q,\emptyset,W) $ and we denote it by $ \cc(Q,W) $.
\end{Def}

 Let $\bmgamma(\overline{Q},\overline{W})$ (or denote by $\overGamma$) is the Ginzburg algebra associated with quiver with potential $(\overline{Q},\overline{W})$. By \cite[Proposition 7.8]{Wu2023}, we have the following homotopy cofiber sequence of dg categories
\[ \bm\Pi_2(F) \rightarrow \bmgamma(Q, F, W) \rightarrow\bmgamma(\overline{Q},\overline{W}).\]

The $G$-action on $(Q,F,W)$ also induces an action on $\overGamma$ by dg algebra automorphisms. 
\begin{Def}\rm\label{Def:Equivariant relative}
	The \emph{G-equivariant relative cluster category} $ \cc(\bmgamma(Q,F,W)*G) $ (or denoted by $ \cc(\bmgamma*G) $) of $ (Q,F,W) $ is defined as the idempotent completion of the Verdier quotient of triangulated categories $$ \mathrm{per}(\bm\Gamma*G)/\mathrm{pvd}(\overGamma*G),$$
	where we view $\mathrm{pvd}(\overGamma*G)$ as a triangulated subcategory of $\per(\bmgamma*G)$ through $$\mathrm{pvd}(\overGamma*G)\simeq(\pvd_{e}(\bmgamma))^{G}\hookrightarrow\per(\bmgamma)^{G}\simeq\per(\bmgamma*G).$$
\end{Def}

Let $(Q_{G},F_{G},W_G)$ be the ice quiver constructed in Definition \ref{Def: Q_{G} and F_{G}} and $\bmgamma_G$ the associated 3-dimensional Ginzburg dg algebra. Denote by $ e_{G}=\sum_{i\in F_{G}}e_{i} $. By Corollary \ref{Cor: stable derived Morita}, we see that $ \cc(\bmgamma*G) $ is equivalent to $\cc(\bmgamma_G)=\per(\bmgamma_G)/\pvd_{e_{G}}(\bmgamma_G)$, i.e. the relative cluster category of $(Q_{G},F_{G},W_G)$.

	If $ F=\emptyset $, the \emph{G-equivariant cluster category} associated with $ (Q,W) $ is given as $$ \cc(\bmgamma(Q,W)*G)=\per(\bmgamma(Q,W)*G)/\pvd(\bmgamma(Q,W)*G),$$
	c.f. \cite{Amiot-Plamondon2018,Paquette-Schiffler2019}.

By Proposition~\ref{Prop:equ of cats} and commutative diagram \ref{diagram:pvd and per}, the action of $ G $ on $ \bm{\Gamma} $ induces an action on relative cluster category $ \cc(\bmgamma) $.
\begin{Prop}\label{Prop:skew-relative-cluster} We have the following equivalence of triangulated categories
	$$ \cc(\bmgamma*G)\iso\cc(\bmgamma)^{G}.$$
\end{Prop}
\begin{proof}
	The relative cluster category $\cc(\bmgamma)$ has a canonical dg-enhancement. By \cite[Corollary 6.10]{Elagin2014}, the corresponding equivariant category $\cc(\bmgamma)^{G}$ is still a triangulated category. 
	
	The shift functor on \( \cc(\bmgamma)^{G} \) is defined as$\colon$ 
	on objects 
	\[
	\Si(X, (\theta^g))\coloneqq (\Si X, (\Si\theta^g));
	\]
	on morphisms in \( \cc(\bmgamma)^{G} \) shift is the same as on morphisms in \( \cc(\bmgamma) \).
	A triangle 
	\[
	(X_1, (\theta_1^g)) \xrightarrow{\alpha} (X, (\theta^g)) \xrightarrow{\beta} (X_2, (\theta_2^g)) \xrightarrow{\gamma} \Si(X_1, (\theta_1^g))
	\]
	in \( \cc(\bmgamma)^{G} \) is distinguished if and only if the triangle
	\[
	X_1 \xrightarrow{\alpha} X \xrightarrow{\beta} X_2 \xrightarrow{\gamma} \Si X_1
	\]
	is distinguished in \( \cc(\bmgamma) \). Then it is easy to see that the quotient functor $\per\bmgamma\ra\cc(\bmgamma)$ induces a triangulated functor $$\Phi\colon\per(\bmgamma)^{G}/\pvd_{e}(\bmgamma)^{G}\longrightarrow\cc(\bmgamma)^{G}.$$
Let $\cs$ be the subcategory of $\per\bmgamma$ formed by the modules $S_{i}$ associated with unfrozen vertices $i\in Q_{0}\setminus F_{0}$. Then $\cs$ is stable under the action of $G$.

Consider the following subcategory of $\per\bmgamma$
\[
\cw= (\Sigma^{\geqslant0} \cs)^\perp \cap {}^\perp (\Sigma^{\leqslant0}\cs).
\]
By \cite[Proposition 3.14]{KellerWu2023}, the following composition
$$\Psi\colon\cw\hookrightarrow\per\bmgamma\rightarrow\cc(\bmgamma)$$
induces a $k$-linear equivalence
$$\Psi\colon\cw\iso\cc(\bmgamma).$$ 
Note that $\cw$ is stable under the action of $G$ and the functor $\Psi$ mentioned above is $G$-equivariant. By \cite[Lemma 2.1]{CCR2020}, the functor $\Psi$ also induces an equivalence of $k$-linear categories
$$\cw^{G}\iso\cc(\bmgamma)^{G}.$$
By Lemma \ref{Lem:goup action on perp} below, $\cw^{G}$ is equivalent to $(\Sigma^{\geqslant0} \cs^{G})^\perp \cap {}^\perp (\Sigma^{\leqslant0}\cs^{G})\subseteq\per(\bmgamma)^{G}.$

Let $(Q_{G},F_{G},W_G)$ be the ice quiver constructed in Definition \ref{Def: Q_{G} and F_{G}} and $\bmgamma_G$ the associated 3-dimensional Ginzburg dg algebra. Let $\cs'$ be the additive subcategory of $\per\bmgamma_G$ formed by the simple modules $S_{i}$ associated with unfrozen vertices $i\in Q_{G,0}\setminus F_{G,0}$. Consider the following subcategory of $\per\bmgamma_G$
\[
\cw'= (\Sigma^{\geqslant0} \cs')^\perp \cap {}^\perp (\Sigma^{\leqslant0}\cs').
\]
Again by \cite[Proposition 3.14]{KellerWu2023}, the following composition
$$\Psi'\colon\cw'\hookrightarrow\per(\bmgamma_G)\simeq\per(\bmgamma*G)\simeq\per(\bmgamma)^{G}\rightarrow\per(\bmgamma)^{G}/\pvd_{e}(\bmgamma)^{G}$$
induces a $k$-linear equivalence
$$\Psi'\colon\cw'\iso\per(\bmgamma)^{G}/\pvd_{e}(\bmgamma)^{G}.$$ 

The equivalence $\per(\bmgamma_G)\simeq\per(\bmgamma*G)\simeq\per(\bmgamma)^{G}$ induces a $k$-linear equivalence $\cs'\simeq\cs^{G} $. Hence $\cw'$ is equivalent to $\cw^{G}$ and we have the following commutative diagram
\[
\begin{tikzcd}
	\cw'\arrow[dr,"\simeq",swap]\arrow[r,hookrightarrow]&\per(\bmgamma_G)\simeq\per(\bmgamma)^{G}\arrow[r]&\per(\bmgamma)^{G}/\pvd_{e}(\bmgamma)^{G}\arrow[r]&\cc(\bmgamma)^{G}\\
	&\cw^{G}\arrow[u,hook]\arrow[urr,"\simeq",swap].
\end{tikzcd}
\]
Therefore $\Phi\colon\per(\bmgamma)^{G}/\pvd_{e}(\bmgamma)^{G}\longrightarrow\cc(\bmgamma)^{G}$ is an equivalence of triangulated categories.

\end{proof}

\begin{Lem}\label{Lem:goup action on perp}
	Let $\ca$ be a $k$-linear additive category and $\cb$ a full additive subcategory. Let $G$ be a finite group acting on $\ca$ such that $\cb$ is stable under this action. Then we have the following identities of $k$-linear categories
	$$(\cb^{\perp})^{G}=(\cb^{G})^{\perp}$$
	$$(^{\perp}\cb)^{G}=\,^{\perp}(\cb^{G}).$$
\end{Lem}
\begin{proof}
	It is clear that we have an inclusion $(\cb^{\perp})^{G}\subseteq(\cb^{G})^{\perp}$. Let $(X,\alpha)$ be an object of $(\cb^{G})^{\perp}$, i.e. we have $\Hom_{\ca^{G}}((Y,\theta),(X,\alpha))=0$ for any object $(Y,\theta)$ in $\cb^{G}$. We will show that $X$ belongs to $\cb^{\perp}$.
	
	Recall that the forgetful functor $U\colon\ca^{G}\ra\ca$ admits a left adjoint $K\colon\ca\ra\ca^{G}, X\mapsto(\oplus_{h\in G}\,^{h}X,\mathrm{Id})$, where $\mathrm{Id}_{g}\colon\,^{g}(\oplus_{h\in G}\,^{h}X)\ra\oplus_{h\in G}\,^{h}X$ is the unique structural isomorphism for any $g\in G$. For each object $B\in\cb$, we have
	\begin{equation*}
		\begin{split}
			\Hom_{\ca}(B,X)&=\Hom_{\ca}(B,U(X,\alpha))\\
			&\simeq\Hom_{\ca^{G}}(K(B),(X,\alpha)).
			\end{split}
	\end{equation*}
	It is clear that $K(B)$ lies in $\cb^{G}$. Therefore $\Hom_{\ca}(B,X)$ vanishes.
	 Hence we see that $X$ belongs to $\cb^{\perp}$, and so $(X,\alpha)$ lies in $(\cb^{\perp})^{G}$. The proof for the second identity is similar.
\end{proof}

\bigskip

Let $ \pr^{F}_{\cd}\bmgamma $ be the following subcategory of $ \pr_{\cd}\bmgamma $
$$ \{\cone(X_{1}\xrightarrow{f} X_{0})\,|\,X_{i}\in \add(\bmgamma)\,\text{and $ \Hom_{\cd}(f, I) $ is surjective for any object $ I\in\cp=\add(e\bmgamma) $}\}. $$
Then we have $ \pr^{F}_{\cd}\bmgamma=\pr_{\cd}\bmgamma\cap^{\perp}\!(\Si^{>0}\cp)\cap(\Si^{<0}\cp)^{\perp} $, where $ \cp=\add(e\relGammabf) $ (see~\cite[pp 12]{KellerWu2023}).

Dually, we define $ \copr_{\cd}^{F}\bmgamma $ as the following subcategory of $ \copr_{\cd}\bmgamma $
$$ \{\Si^{-1}\cone(X_{0}\xrightarrow{f} X_{1})\,|\,X_{i}\in \add(\bmgamma)\,\text{and $ \Hom_{\cd}(P,f) $ is surjective for any object $ P\in\cp=\add(e\bmgamma) $}\}. $$
And we have $ \copr^{F}_{\cd}\bmgamma=\copr_{\cd}\bmgamma\cap\cz $.
Similarly, we define subcategories $$ \pr_{\cc}^{F}\bmgamma=\pr_{\cc}\bmgamma\cap\cy $$ and 
$$ \copr_{\cc}^{F}\bmgamma=\copr_{\cc}\bmgamma\cap\cy $$ of $ \cc $, where $$ \cy=\,^{\perp}\!(\Si^{>0}\cp)\cap(\Si^{<0}\cp)^{\perp}\subseteq\cc. $$
\begin{Def}\rm\cite[Definition 3.21]{KellerWu2023}\label{Def: Higgs cat}
	The \emph{Higgs category} $ \ch(Q,F,W) $ (or denoted by $ \ch(\bmgamma) $) is defined as the full subcategory of $ \pr^{F}_{\cc}\bm{\Gamma}\cap\copr^{F}_{\cc}\bmgamma $ whose objects are those $ X $ such that $ \Hom_{\cc}(\Si^{-1}\bmgamma,X) $ is finite-dimensional.
\end{Def}

Under the assumption \ref{ass2}, we have the following results.
\begin{Thm}\cite[Theorem 4.14, Theorem 4.18]{KellerWu2023}\label{Thm: Thm of KellerWu}
	\begin{itemize}
		\item The Higgs category $\mathcal{H}$ is equal to the full subcategory $\mathcal{E}$ of $\mathcal{C}(Q, F, W)$ defined by
		\[
		\mathcal{E} = \{ X \in \cc(\bmgamma) \mid \operatorname{Hom}_{\cc(\bmgamma)}(X, \Sigma^{>0} \mathcal{P}) = 0 = \operatorname{Hom}_{\cc(\bmgamma)}(\Sigma^{<0} \mathcal{P}, X) \}.
		\]
		\item The Higgs category $\ch(\bmgamma)$ is a Frobenius extriangulated category with projective-injective objects $\cp=\add(e\bmgamma)$. Its stable category $\underline{\ch}=\ch/[\cp]$ is equivalent to $\cc(\overline{Q},\overline{W})$. Moreover, $\bmgamma$ a canonical cluster-tilting object of $\ch$ with endomorphism algebra $J(Q,F,W)$.
	\end{itemize}
\end{Thm}

The $G$-equivariant relative cluster category $\cc(\bmgamma*G)$ is equivalent to $\cc(\bmgamma_G)=\per(\bmgamma_G)/\pvd_{e'}(\bmgamma_G)$ (see Definition \ref{Def:Equivariant relative}). We make the following definition.
\begin{Def}\rm\label{Def: Equivariant Higgs}\cite{Wu2023}
	The \emph{G-equivariant Higgs category} is defined to be $\ch(\bmgamma_G)$, i.e. the following full subcategory of $\cc(\bmgamma_G)$
	\[
	\{ X \in \cc(\bmgamma_G) \mid \operatorname{Hom}_{\cc(\bmgamma)}(X, \Sigma^{>0} \mathcal{P}') = 0 = \operatorname{Hom}_{\cc(\bmgamma))}(\Sigma^{<0} \mathcal{P}', X) \},
	\]
where $\cp'=\add(e_{G}\bmgamma_G)$, i.e. the additive subcategory of $\cc(\bmgamma_G)$ generated by $e_{G}\bmgamma_G$. 	
\end{Def}

Since the category $\cp=\add(e\bmgamma)=\add((\sum_{i\in Q_{0}\setminus F_{0}}e_{i})\bmgamma)$ is stable under the $G$-action, by Theorem \ref{Thm: Thm of KellerWu},  the $G$-action on $\cc(\bmgamma)$ induces a $G$-action on Higgs category $\ch(\bmgamma)$.

\begin{Thm}\label{Prop: skew-Higgs}
	The triangle equivalence $\cc(\bmgamma*G)\iso\cc(\bmgamma)^{G}$ in Proposition \ref{Prop:skew-relative-cluster} induces an equivalence of Frobenius extriangulated categories
	$$ \ch(\bmgamma_G)\iso\ch(\bmgamma)^{G}.$$
The category of projective-injective objects of $\ch^{G}$ is $\cp^{G}=\add(e\bmgamma)^{G}$ and $\bmgamma[G]$ is a canonical cluster-titling object of $\ch^{G}$. And $\bmgamma_{G}$ is a basic generator of $\add(\bmgamma_{G})\simeq\add(\bmgamma[G])$.
\end{Thm}
\begin{proof}
For each $h\in G$, the functor $^{h}()\colon\ch(\bmgamma)\ra\ch(\bmgamma)$ is exact. Then $\ch(\bmgamma)^{G}$ is an extriangulated category with $\mathbb{E}$-extensions of the form \[
(X, \psi) \xrightarrow{f} (X', \psi') \xrightarrow{g} (X'', \psi'')
\]
such that 
\[
 X \xrightarrow{f} X' \xrightarrow{g} X''
\]
is an $\mathbb{E}$-extension in $\ch(\bmgamma)$. By a similar argument in \cite[Corollary 2.41]{LD2011}, we see that $\ch(\bmgamma)^{G}$ is a Frobenius extriangulated category with projective-injective objects $\cp^{G}=\add(e\bmgamma)^{G}$. Moreover it is stably 2-Calabi–Yau.

The equivalence $\cc(\bmgamma_G)\simeq\cc(\bmgamma*G)\iso\cc(\bmgamma)^{G}$ induces a $k$-linear equivalence $\cp'\iso\cp^{G}$. By Lemma \ref{Lem:goup action on perp}, we get an equivalence $\ch(\bmgamma_{G})\iso\ch^{G}$ of $k$-linear categories. It is clear that this equivalence preserves extriangulated structures. Hence it is an equivalence of Frobenius extriangulated categories.

\end{proof}

\begin{Cor}
	The quotient functor $\ch\ra\underline{\ch}=\ch/[\cp]$ induces an equivalence of triangulated categories
	$$\underline{(\ch^{G})}=\ch^{G}/[\cp^{G}]\iso(\underline{\ch})^{G}.$$
\end{Cor}

\begin{proof}
	We have the following commutative diagram
	\[
	\begin{tikzcd}
		\cc(Q_{G},F_{G},W_{G})\arrow[rr,"\simeq"]\arrow[dd]&&\cc(Q,F,W)^{G}\arrow[dd,dashed]\\
		&\ch(\bmgamma*G)\arrow[ul,hook]\arrow[rr,"\simeq\qquad\qquad"]\arrow[dl,two heads]&&\ch^{G}\arrow[ul,hook]\arrow[dl,two heads]\\
		\cc(\overline{Q_{G}},\overline{W_{G}})\arrow[rr,"\simeq"]&&\cc(\overline{Q},\overline{W})^{G}.
	\end{tikzcd}
	\]
Hence we get the following equivalences of triangulated categories
$$\ch(\bmgamma*G)/[\cp']\iso\cc(\overline{Q_{G}},\overline{W_G})\iso\cc(\overline{Q},\overline{W})^{G}\simeq(\ch/[\cp])^{G}.$$	
By Proposition \ref{Prop: skew-Higgs} above, the triangulated category $\underline{\ch^{G}}$ is equivalent to $(\ch/[\cp])^{G}=(\underline{\ch})^{G}$.
\end{proof}

\subsection{$ G $-stable cluster-tilting subcategories}

\begin{Def}\rm
	Let $(\cc,\mathbb{E},\mathfrak{s})$ be an extriangulated category. A subcategory $\ct$ of $\cc$ is said to be \emph{cluster-tilting subcategory} if it satisfies the following conditions$\colon$
	\begin{itemize}
		\item $\ct$ is functorially finite in $\cc$;
		\item $X\in\ct$ if and only if $\mathbb{E}(X,\ct)=0$;
		\item $X\in\ct$ if and only if $\mathbb{E}(\ct,X)=0$.
	\end{itemize}

A subcategory $\ccr$ of $\cc$ is said to be \emph{rigid} if there are no non trivial extensions between its objects. If moreover every rigid category $\ccr'$ containing $\ccr$ is equal to $\ccr$ , then we say that $\ccr$ is \emph{maximal rigid}.
	
\end{Def}

The action of $G$ on $\cd(\bmgamma)$ stabilizes $\per(\bmgamma)$ and $\pvd_{e}(\bmgamma)$. It induces a strict action of $G$ on $\cc(\bmgamma)$ by strict automorphisms of triangulated categories and also on $\ch$ by strict automorphisms of Frobenius extriangulated categories.


It is easy to see that the adjoint pair of triangle functors
\[
\begin{tikzcd}
	\mathrm{Ten}=-\lten_{\bmgamma}(\bmgamma*G)\colon\cd(\bmgamma)\arrow[r,shift left=1]&\cd(\bmgamma*G)\colon\mathrm{Res}\arrow[l,shift left=1]
\end{tikzcd}
\]

induces the following adjoint pair of extriangulated functors
\[
\begin{tikzcd}
	\mathrm{Ten}\colon\ch(\bmgamma)\arrow[r,shift left=1]&\ch(\bmgamma*G)\colon\mathrm{Res}\arrow[l,shift left=1].
\end{tikzcd}
\]

\begin{Def}\rm
	Let $\cc$ be an extriangulated category, we define $\mathfrak{Add}(\cc)$ to be the class of all full $k$-linear subcategories of $\cc$ which are stable under isomorphisms and direct summands. If $E$ is a collection of objects of $\cc$ , one denotes by $\add(E)$ the smallest category of $\mathfrak{Add}(\cc)$ containing $E$. 
	
	A category $\ct \in\mathfrak{Add}(\cc)$ is said to be finitely generated if $\ct$ is of the form $\add(M)$ for some object $M\in \cc$ . The subclass of $\mathfrak{Add}(\cc)$ consisting in all finitely generated categories will be denoted by $\mathfrak{add}(\cc)$. If $H$ is an exact functor from $\cc$ to another extriangulated category $\cc'$ and $\ct\in \mathfrak{Add}(\cc)$, $H(\ct)$ will denote $\add({H(X)\,|\,X\in\ct } )$.
	
	If a finite group $G$ acts on $\cc$, we denote by  $\mathfrak{Add}(\cc)^{G}$ (resp. $\mathfrak{add}(\cc)^{G}$ ) the class of elements of $\mathfrak{Add}(\cc)$ (resp. $\mathfrak{add}(\cc)$) which are $G$-stable objects of $\mathfrak{Add}(\cc)$ (resp. $\mathfrak{add}(\cc)$).
	
	If $\cc$ is an $\cm$-module category for some monoidal category $\cm$, we denote by  $\mathfrak{Add}(\cc)^{\cm}$ (resp. $\mathfrak{add}(\cc)^{\cm}$ ) the class of elements of $\mathfrak{Add}(\cc)$ (resp. $\mathfrak{add}(\cc)$) which are sub-$\cm$-module categories of $\cc$.
	
\end{Def}

\begin{Prop}\label{Prop:CTO equ1}
	The adjunction $(\mathrm{Ten},\mathrm{Res})$ induces reciprocal bijections between$\colon$
	\begin{itemize}
		\item $\mathfrak{Add}(\ch)^{G}$ and $\mathfrak{Add}(\ch(\bmgamma*G))^{\mathrm{Ten}\circ\mathrm{Res}}=\{\ct\in\mathfrak{Add}(\ch(\bmgamma*G))\,|\,\mathrm{Ten}\circ\mathrm{Res}(\ct)\subseteq\ct\}$;
		\item the set of rigid $\ct\in\mathfrak{Add}(\ch)^{G}$ and the set of rigid $\ct\in\mathfrak{Add}(\ch(\bmgamma*G))^{\mathrm{Ten}\circ\mathrm{Res}}$;
		\item the set of maximal $G$-stable rigid $\ct\in\mathfrak{Add}(\ch)^{G}$ and the set of maximal rigid $\ct\in\mathfrak{Add}(\ch^{G})^{\mathrm{Ten}\circ\mathrm{Res}}$;
		\item the set of cluster-tilting $\ct\in\mathfrak{Add}(\ch)^{G}$ and the set of cluster-tilting $\ct\in\mathfrak{Add}(\ch(\bmgamma*G))^{\mathrm{Ten}\circ\mathrm{Res}}$.
\end{itemize}
Moreover, all these bijections restrict to bijections between the corresponding finitely generated classes.
\end{Prop}

\begin{proof}
	The proof is similar in spirit to \cite[Corollary 5.2.2]{LeMeur2020}.
\end{proof}

\section{Categorification of skew-symmetrizable cluster algebras with coefficients}\label{Section6}

The aim of this section is to generalize Demonet's result \cite{LD2011} to the setting of ice quivers with potentials. Let $(Q,F,W)$ be an ice quiver with potential satisfies Assumption ~\ref{ass2}. Let $G$ be a finite group acting on $(Q,F,W)$, where the $G$-action satisfies Assumption~\ref{ass1}. Additionally, we assume that $\charc(k)\nmid\car(G)$.

\subsection{A $\mod (k[G])$-linear Structure on the Equivariant Higgs Category}
We denote by $k[G]$ the group algebra of $G$. It is a Hopf algebra, hence $\mod (k[G]) $ is a monoidal category. An object of $\mod (k[G]) $ will be denoted by $(V,r)$, where $V$ is a finite dimensional $k$-vector space, and $r\colon G\ra \mathrm{GL}(V)$ a group homomorphism.

\begin{Prop}\cite[Proposition 2.12]{LD2011}
	The equivariant Higgs category $\ch^{G}$ is a $\mod (k[G]) $-module category in a natural way.
\end{Prop}
%
%

Let $(X,\psi)$ and $(Y,\chi)$ be objects of $\ch^{G}$, Let 
$$\bmhom_{\ch^{G}}((X,\psi),(Y,\chi))\coloneqq\Hom_{\ch}(X,Y).$$

If \( g \in G \) and \( f \in \bmhom_{\ch^{G}}((X,\psi),(Y,\chi)) \), define \( gf \in \bmhom_{\ch^{G}}((X,\psi),(Y,\chi)) \) by the following commutative diagram:
\[
\begin{tikzcd}
	X\arrow[r,"gf"]\arrow[d,"\psi_g^{-1}"]&Y\arrow[d,"\chi^{-1}_g"]\\
	^{g}X\arrow[r,"^{g}f"]&^{g}Y.
\end{tikzcd}
\]

This defines a structure of a $\mod(k[G])$ structure on $\bmhom_{\ch^{G}}((X,\psi),(Y,\chi))$, see \cite[Proposition 2.17]{LD2010}.

Recall from Section \ref{Section4:Equivariant categories} that the $G$-action on the Higgs category $\ch(\bmgamma)$  induces the following adjoint pair of extriangulated functors (also see \cite[Section 2.4]{LD2011})
\[
\begin{tikzcd}
 K\colon\ch(\bmgamma)\arrow[r,shift left=1]&\ch(\bmgamma)^{G}\arrow[l,shift left=1]\colon U,
\end{tikzcd}
\]
where $U$ is the forgetful functor and $K$ maps an object $X$ to 
$X[G]=(\oplus_{h\in G}\,^{h}X,\mathrm{Id}).$ Here $\mathrm{Id}_{g}\colon\oplus_{h\in G}\,^{h}X\ra\,^{g}(\oplus_{h\in G}\,^{h}X)$ is the identity map for any $g\in G$. 

\begin{Prop}\cite[Proposition 2.22]{LD2011}\label{Prop:U-kG iso}
There is an isomorphism of functors from \( \ch(\bmgamma)^{G} \) to itself:
\[
(U -)[G] \cong k[G] \otimes -
\]
where \( U\colon \ch(\bmgamma)^{G} \rightarrow \ch(\bmgamma) \) is the forgetful functor and \( k[G] \) denotes the regular representation of \( G \).
\end{Prop}

\begin{Def}\rm
	A subcategory $\ct$ of $\ch(\bmgamma)$ (resp. of $\ch^{G}$) is said to be \emph{rigid} if there are no non trivial extensions between its objects. If moreover every rigid category $\ct'$ containing $\ct$ is equal to $\ct$ , then we say that $\ct$ is \emph{maximal rigid}.
\end{Def}

\begin{Prop}\cite[Proposition 3.5]{LD2011}\label{Prop:G-stabe and K[G]-stable}
The adjunction $(K,U)$ induces reciprocal bijections between$\colon$
\begin{itemize}
	\item $\mathfrak{Add}(\ch)^{G}$ and $\mathfrak{Add}(\ch^{G})^{\mod(k[G])}$;
	\item the set of rigid $\ct\in\mathfrak{Add}(\ch)^{G}$ and the set of rigid $\ct\in\mathfrak{Add}(\ch^{G})^{\mod(k[G])}$;
	\item the set of maximal $G$-stable rigid $\ct\in\mathfrak{Add}(\ch)^{G}$ and the set of maximal $\mod(k[G])$-stable rigid $\ct\in\mathfrak{Add}(\ch^{G})^{\mod(k[G])}$;
	\item the set of cluster-tilting $\ct\in\mathfrak{Add}(\ch)^{G}$ and the set of cluster-tilting $\ct\in\mathfrak{Add}(\ch^{G})^{\mod(k[G])}$.
\end{itemize}
Moreover, all these bijections restrict to bijections between the corresponding finitely generated classes.
\end{Prop}
\begin{proof}
	The proof is similar in spirit to \cite[Proposition 3.5]{LD2011}.
\end{proof}
\bigskip
%
%
%

%



\begin{Def}\rm
	Let $\cd$ be a $\mod(k[G])$-stable cluster-tilting subcategory of  $\ch(\bmgamma)^{G}$ and let $X\in\cd$ be indecomposable. A \emph{$\mod(k[G])$-loop} of $\cd$ at $X$ is an irreducible morphism $X\ra X'$ of $\cd$ where $X'\in\add(k[G]\ten X)$ is indecomposable. A \emph{$\mod(k[G])$-2-cycle} of $\cd$ at $X$ is a pair of irreducible morphisms $X\ra Y$ and $Y\ra X'$ of $\cd$ where $X'\in\add(k[G]\ten X)$ is indecomposable.
\end{Def}

\begin{Def}\rm
	Let $\ct$ be a $G$-stable cluster tilting subcategory of  $\ch(\bmgamma)$ and let $X\in\ct$ be indecomposable. A \emph{$G$-loop} of $\ct$ at $X$ is an irreducible morphism $X\ra\, ^{g}X$ of $\cd$ where $g\in G$. A \emph{$G$-$2$-cycle} of $\ct$ at $X$ is a pair of irreducible morphisms $X\ra Y$ and $Y\ra\,^{g}X$ of $\cd$ where $g\in G$.
\end{Def}

\begin{Lem}\cite[Lemma 3.24]{LD2011}
	Let $\cd$ be a $\mod(k[G])$-stable cluster tilting subcategory of  $\ch(\bmgamma)^{G}$. Let $X\in\cd$ be indecomposable and $X'$ be a direct summand of $U(X)$. Then $\cd$ has no $\mod(k[G])$-loops (resp. $\mod(k[G])$-2-cycles) at $X$ if and only if $U(\cd)$ has no $G$-loops (resp. $G$-2-cycles) at $X'$.
	
\end{Lem}

%
Since $\add(\bmgamma)$ is a canonical $G$-stable cluster tilting subcategory of  $\ch(\bmgamma)$, by Proposition \ref{Prop:G-stabe and K[G]-stable}, the subcategory $K(\add(\bmgamma))=\add(K(\bmgamma))=\add(\bmgamma[G])$ is a canonical $\mod(k[G])$-stable cluster tilting subcategory of  $\ch(\bmgamma)^{G}$ and
\begin{equation*}
	\begin{split}
		U(\add(K(\bmgamma)))&=U(K(\add(\bmgamma)))\\
		&=\add(\oplus_{h\in G}\,^{h}\bmgamma)\\
		&=\add(\bmgamma).
	\end{split}
\end{equation*} 
Let $(Q_{G},F_{G})$ be the ice quiver constructed in Definition \ref{Def: Q_{G} and F_{G}}.  
\begin{Cor}
The $\mod(k[G])$-stable cluster tilting subcategory $\add(K(\bmgamma))$	has no $\mod(k[G])$-loops (resp. $\mod(k[G])$-2-cycles) at each indecomposable object if and only if $Q_{G}$ has no loops (resp. 2-cycles). 
\end{Cor}

\begin{proof}
By the Lemma above, $\add(K(\bmgamma))$	has no $\mod(k[G])$-loops (resp. $\mod(k[G])$-2-cycles) at each indecomposable object if and only if $\add(\bmgamma)$ has no $G$-loops (resp. $G$-2-cycles). By the construction of $Q_{G}$, this is equivalent to $Q_{G}$ having no loops (resp. 2-cycles).
	
\end{proof}

\subsection
{Exchange matrices and cluster characters}

In this subsection, we assume that the characteristic of the field $k$ is zero and $Q$ has no loops or 2-cycles. 

\begin{Def}\rm\cite{Dupont2009}
	We say that the $G$-action on $Q$ is \emph{admissible} if $Q$ has no $G$-loops or $G$-2-cycles. The pair $(Q,G)$ is then called
	an \emph{admissible pair}.
\end{Def}

We assume that the $G$-action on $Q$ is admissible and the $G$-invariant potential $W$ is non-degenerate~\cite[Definition 7.2]{DWZ2008}. The Gabriel quiver of $\bmgamma$ in $\ch(\bmgamma)$ is isomorphic to $Q$. Suppose that the number of vertices of $Q$ is $n$. One denotes by $\bmgamma_{1}, \bmgamma_{2},\ldots, \bmgamma_{n}$ the indecomposable objects of $\add(\bmgamma)$ up to isomorphism, the $\bmgamma_{i}$ for $i\in[r+1,n]$ being the projective-injective objects. There is a canonical bijection between the sets $Q_{0}$ and $\{\bmgamma_{1},\cdots,\bmgamma_{n}\}$. The action of $G$ on $\bmgamma$ induces an action on $I=[1,n]$. Denote by $I_{uf}=\{1,\cdots,r\}$ and by $I_{f}=\{r+1,\cdots,n\}$.

If $i\in I$, we denotes by $\bm{i}=G\cdot i$ the corresponding orbit set. Define $\overline{I}=\{\bm{1},\cdots\bm{m}\}$ to be the set of these orbit sets. Denote by $\overline{I}_{uf}=\{\bm{1},\cdots,\bm{s}\}$ the set of unfrozen equivalence classes and by $\overline{I}_{f}=\{\bm{s+1},\cdots,\bm{m}\}$ the set of frozen equivalence classes.
%
%

Let $\bm{i}\in\overline{I}_{uf}$. Assume that $\bm{i}=\{i_{1},\cdots,i_{s}\}$. Since $Q$ has no $G$-loops, we have $\mu_{i_{u}}\circ\mu_{i_{v}}=\mu_{i_{v}}\circ\mu_{i_{u}}$ for each $i_{u},i_{v}$ in $\bm{i}$. We define the \emph{orbit mutation} $\mu_{\bm{i}}(Q)$ of $Q$ at $\bm{i}$ as $\mu_{i_{s}}\circ\mu_{i_{s-1}}\circ\cdots\circ\mu_{i_{1}}(Q)$.

\begin{Def}\rm\cite[Definition 2.19]{Dupont2009}
	We say that the $G$-action on $Q$ is \emph{stable} if for any finite sequence of unfrozen $G$-orbits $(\bm{k_{1}},\ldots,\bm{k_{r}})$, each of the pairs $(\mu_{\bm{k_{1}}}(Q), G), (\mu_{{\bm{k_{2}}}} \circ \mu_{\bm{k_{1}}}(Q), G), \dots, (\mu_{\bm{k_{n}}}\circ \cdots \circ \mu_{\bm{k_{1}}}(Q), G)$ is admissible.
\end{Def}

\begin{Rem}
	If the relative Ginzburg dg algebra $\bmgamma(Q, F, W)$ is concentrated in degree zero, then the Higgs category $\ch$ is a usual Frobenius exact category, as shown in \cite[Theorem 4.18]{KellerWu2023}. Then we can use a similar proof to that of \cite[Theorem 3.33]{LD2011} to show that the admissible $G$-action on $Q$ is stable.
\end{Rem}

Now we assume that $(Q,G)$ is an admissible pair.

\begin{Def}\cite[Definition 3.36]{LD2011}\rm\label{Def:skew-symmetrizable}
	Let $\bm{i}\in\overline{I}$ and $\bm{j}\in\overline{I}_{uf}$. We define
	$$b_{\bm{i},\bm{j}}=\frac{\sharp\{a\in Q_{1}\,|\,s(a)\in\bm{i},t(a)\in\bm{j}\}-\sharp\{a\in Q_{1}\,|\,s(a)\in\bm{j},t(a)\in\bm{i}\}}{\sharp(\bm{j})}.$$ Denote by $B(\bmgamma)=(b_{\bm{i},\bm{j}})_{i\in\overline{I},j\in\overline{I}_{uf}}$ the matrix having these entries. It is called the \emph{exchange matrix of $\bmgamma$ with respect to the group action $G$}.
\end{Def}

\begin{Rem}
	\begin{itemize}
		\item As $Q$ has no $G$-2-cycles, we have $\sharp\{a\in Q_{1}\,|\,s(a)\in\bm{i},t(a)\in\bm{j}\}=0$ or $\sharp\{a\in Q_{1}\,|\,s(a)\in\bm{j},t(a)\in\bm{i}\}=0$.
		\item As $Q$ has no $G$-loops, we have $b_{\bm{i},\bm{i}}=0$ for each $\bm{i}\in\overline{I}$.
		\item It is easy to see that $$b_{\bm{i},\bm{j}}=\sharp\{a\in Q_{1}\,|\,s(a)\in\bm{i},t(a)=j\}-\sharp\{a\in Q_{1}\,|\,s(a)=j,t(a)\in\bm{i}\} $$ for each $\bm{i},\bm{j}\in\overline{I}$. Hence $B(\bmgamma)$ has integer coefficients.
		\item The exchange matrix is clearly skew-symmetrizable with with symmetrizer $D=(d_{\bm{i}})_{\bm{i}\in\overline{I}_{uf}}$, where $d_{\bm{i}}=\sharp(\mathrm{stab}(i))$. Here $\mathrm{stab}(i)$ denotes the stabilizer of $i$ for the $G$-action.
	\end{itemize}
\end{Rem}



 Denote by $\tilde{B}(\bmgamma)=(\tilde{b}_{i,j})_{i\in I,j\in I_{uf}}$ the usual exchange matrix of $Q$, i.e. $$\tilde{b}_{i,j}=\sharp\{a\in Q_{1}\,|\,s(a)=i,t(a)=j\}-\sharp\{a\in Q_{1}\,|\,s(a)=j,t(a)=i\}.$$

\begin{Lem}\label{Lem:two matrix relation}
	Let $\bm{i},\bm{j}\in\overline{I}$. For each $j\in\bm{j}$, we have 
	$b_{\bm{i},\bm{j}}=\sum_{k\in\bm{i}}\tilde{b}_{k,j}$.
\end{Lem}
\begin{proof}
	It is clear from the definition.
\end{proof}

\begin{Ex}\label{Ex:type A3}
Let $(Q,F)$ be the following ice quiver
\[
\begin{tikzcd}
	1\arrow[r]&2&3\arrow[l]\\
	\color{blue}\boxed{4}\arrow[u]&\color{blue}\boxed{5}\arrow[u]&\color{blue}\boxed{6}\arrow[u]
\end{tikzcd}
\]
on which $G=\mathbb{Z}/2\mathbb{Z} $ acts in the only non-trivial possible way. Then $\overline{I}=Q_{0}/G=\{\bm{1},\bm{2},\textcolor{blue}{\bm{4}},\textcolor{blue}{\bm{5}}\} $ and we obtain matrix $B(\bmgamma)$
\[
\begin{bmatrix}
	0&2\\
	-1&0\\
	1&0\\
	0&1
\end{bmatrix}
\]
which is the initial exchange matrix of cluster algebra type $B_{2}$ (or $C_{2}$) with principal coefficients. The corresponding valued ice quiver is 
\[
\begin{tikzcd}
	\bm{1}\arrow[r,"{(2,1)}"]&\bm{2}\\
	\color{blue}\boxed{\bm{4}}\arrow[u]&\color{blue}\boxed{\bm{5}}\arrow[u]\color{black}.
\end{tikzcd}
\]

\end{Ex}

\begin{Ex}\label{Ex:typeA5}
Let $(Q,F)$ be the following ice quiver
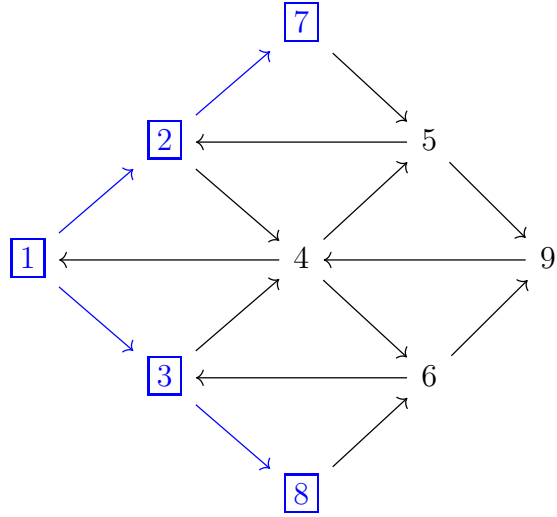
\[
\begin{tikzcd}
	&&\color{blue}\boxed{7}\arrow[dr]&\\
	&\color{blue}\boxed{2}\arrow[ur,blue]\arrow[dr]&&5\arrow[ll]\arrow[dr]&\\
	\color{blue}\boxed{1}\arrow[ur,blue]\arrow[dr,blue]&&4\arrow[dr,swap]\arrow[ur]\arrow[ll]&&9\arrow[ll]\\
	&\color{blue}\boxed{3}\arrow[dr,blue]\arrow[ur]&&6\arrow[ll]\arrow[ur]&\\
	&&\color{blue}\boxed{8}\arrow[ur]&
\end{tikzcd}
\]
with a $ \mathbb{Z}/2\mathbb{Z}$-action which is the reflection along middle horizontal line. Then $\overline{I}=\{\bm{4},\bm{5},\bm{9},\textcolor{blue}{\bm{1}},\textcolor{blue}{\bm{2}},\textcolor{blue}{\bm{7}}\}$ and the matrix $B(\bmgamma)$ is
\[
\begin{bmatrix}
	0&2&-1\\
	-1&0&1\\
	1&-2&0\\
	-1&0&0\\
	1&-1&0\\
	0&1&0
\end{bmatrix}
\]	
which is the initial exchange matrix of cluster algebra structure on the maximal unipotent subgroup of a Lie group of type $B_{3}$ \cite[Example 4.22]{LD2011}.
\end{Ex}

Let $v$ be an unfrozen vertex of $Q$. Define $\mu^{+}(\bmgamma)=\bigoplus_{i\neq v}\bmgamma_{i}\oplus\bmgamma'_{v}$, where $\bmgamma'_{v}$ is given by the cone of 
$$\bmgamma_{v}\ra\bigoplus_{\alpha\in Q_{1},s(\alpha)=v}\bmgamma_{t(\alpha)}$$
whose components are given by left multiplication by $\alpha$.

Similarly, define $\mu^{-}(\bmgamma)=\bigoplus_{i\neq v}\bmgamma_{i}\oplus\bmgamma''_{v}$, where $\bmgamma''_{v}$ is given by the cocone of 
$$ \bigoplus_{\beta\in Q_{1},t(\beta)=v}\bmgamma_{s(\beta)}\ra\bmgamma_{v}$$
whose components are given by left multiplication by $\beta$. By \cite[Remark 5.5]{Wu20231}, $\mu_{v}^{+}(\bmgamma)$ and $\mu_{v}^{-}(\bmgamma)$ are isomorphic in the Higgs category $\ch$. Then the mutation $\mu_{v}(\bmgamma)$ of $\bmgamma$ at $v$ is denoted as $\mu_{v}^{+}(\bmgamma)$ or $\mu_{v}^{-}(\bmgamma)$. Denote by $\mu^{FZ}_{v}$ the Fomin--Zelevinsky's mutation of quiver or matrix. Then the Gabriel quiver of $\End_{\ch}(\mu_{i}\bmgamma)$ is isomorphic to $\mu^{FZ}_{v}(Q)$ since $W$ is non-degenerate.

For each $\bm{i}\in\overline{I}_{uf}$, assume that $\bm{i}=\{i_{1},\cdots,i_{s}\}$, we define the \emph{orbit mutation} $\mu_{\bm{i}}(\bmgamma)$ of $\bmgamma$ at $\bm{i}$ as $\mu_{i_{s}}\circ\mu_{i_{s-1}}\circ\cdots\circ\mu_{i_{1}}(\bmgamma)$. For each $i_{u},i_{v}$ in $\bm{i}$, since $Q$ has no $G$-loops, we have $\mu_{i_{u}}\circ\mu_{i_{v}}=\mu_{i_{v}}\circ\mu_{i_{u}}$. Hence the object $\mu_{\bm{i}}(\bmgamma)$ is well defined and $G$ also naturally acts on it.

\begin{Rem}
	The orbit mutation of $\bmgamma$ in $\ch$  corresponds to the mutation of $\mod k[G]$-stable cluster tilting subcategory $\add(\bmgamma[G])$ in $\ch^{G}$ (\cite[Theorem 3.35, Proposition 3.40]{LD2011}).
\end{Rem}

\begin{Prop}\cite[Theorem 3.42]{LD2011}
	Let $\bm{i}\in\overline{I}_{uf}$ and assume that $\bm{i}=\{i_{1},\cdots,i_{s}\}$. Then $$B(\mu_{\bm{i}}(\bmgamma))=\mu^{FZ}_{\bm{i}}(B(\bmgamma)).$$
\end{Prop}

\begin{proof}
Let $\tilde{B}=(\tilde{b}_{k,l})$ be the associated skew-symmetric matrix of $Q$. Then the matrix $B(\bmgamma)=(b_{\bm{k},\bm{l}})$ has coefficients
$$b_{\bm{k},\bm{l}}=\sum_{a\in\bm{k}}\tilde{b}_{a,l}.$$

The associated skew-symmetric matrix $\tilde{B}(\mu_{\bm{i}}(\bmgamma))=(\tilde{b}^{s}_{k,l})$ is equal to $\mu_{\bm{i}}(\tilde{B})=\mu_{i_{s}}\circ\mu_{i_{s-1}}\circ\cdots\circ\mu_{i_{1}}(\tilde{B}).$

By induction, one can show that 
\begin{equation*}
	\tilde{b}^{s}_{k,l}=
	\begin{cases}
		-\tilde{b}_{k,	l}&\mbox{if $k$ or $l$$\in\bm{i}$,}\\
		\tilde{b}_{k,l}+\frac{1}{2}\sum_{r=1}^{s}(|\tilde{b}_{k,i_{r}}|\tilde{b}_{i_{r},l}+\tilde{b}_{k,i_{r}}|\tilde{b}_{i_{r},l}|)&\mbox{otherwise.}
	\end{cases}
\end{equation*}

Suppose that $B(\mu_{i}(\bmgamma))=(b^{s}_{\bm{k},\bm{l}})$. By Lemma \ref{Lem:two matrix relation}, we have 
$$b^{s}_{\bm{k},\bm{l}}=\sum_{r\in\bm{k}}\tilde{b}^{s}_{r,l}.$$
If $\bm{k}=\bm{i}$ or $\bm{l}=\bm{i}$, we obtain
$b^{s}_{\bm{k},\bm{l}}=-\sum_{r\in\bm{k}}\tilde{b}_{r,l}.$

Otherwise, if $\bm{k},\bm{l}\neq\bm{i}$ then
\begin{equation*}
	\begin{split}
		b^{s}_{\bm{k},\bm{l}}=&\sum_{s\in\bm{k}}(\tilde{b}_{s,l}+\frac{1}{2}\sum_{r=1}^{s}(|\tilde{b}_{s,i_{r}}|\tilde{b}_{i_{r},l}+\tilde{b}_{s,i_{r}}|\tilde{b}_{i_{r},l}|))\\
		=&\sum_{s\in\bm{k}}(\tilde{b}_{s,l}+\frac{1}{2}\sum_{t\in\bm{i}}(|\tilde{b}_{s,t}|\tilde{b}_{t,l}+\tilde{b}_{s,t}|\tilde{b}_{t,l}|)).
	\end{split}
\end{equation*}

On the other hand, the coefficients $c_{\bm{k},\bm{l}}$ of $\mu^{FZ}_{\bm{i}}(B(\bmgamma))=(c_{\bm{k},\bm{l}})$ are
\begin{equation*}
	c_{\bm{k},\bm{l}}=
	\begin{cases}
		-b_{\bm{k},\bm{l}}&\mbox{if $\bm{k}=\bm{i}$ or $\bm{l}=\bm{i}$,}\\
		b_{\bm{k},\bm{l}}+\frac{1}{2}(|b_{\bm{k},\bm{i}}|b_{\bm{i},\bm{l}}+b_{\bm{k},\bm{i}}|b_{\bm{i},\bm{l}}|)&\mbox{otherwise.}
	\end{cases}
\end{equation*}

Thus, if $\bm{k}=\bm{i}$ or $\bm{l}=\bm{i}$, we have $c_{\bm{k},\bm{l}}=-b_{\bm{k},\bm{l}}=-\sum_{a\in\bm{k}}\tilde{b}_{a,l}=b^{s}_{\bm{k},\bm{l}}$. If $\bm{k},\bm{l}\neq\bm{i}$, we can check that
\begin{equation*}
	\begin{split}
		c_{\bm{k},\bm{l}}&=b_{\bm{k},\bm{l}}+\frac{1}{2}(|b_{\bm{k},\bm{i}}|b_{\bm{i},\bm{l}}+b_{\bm{k},\bm{i}}|b_{\bm{i},\bm{l}}|)\\
		&=\sum_{m\in\bm{k}}(\tilde{b}_{m,l}+\frac{1}{2}\sum_{n\in\bm{i}}(|\tilde{b}_{m,i}|\tilde{b}_{n,l}+\tilde{b}_{m,i}|\tilde{b}_{n,l}|)).
	\end{split}
\end{equation*} 

Then it is not hard to see that $c_{\bm{k},\bm{l}}$ is equal to $b^{s}_{\bm{k},\bm{l}}$.
\end{proof}

\bigskip

For each $ X $ in $ \pr_{\cc(Q,F,W)}\bmgamma $, define the \emph{index with respect to $ \bmgamma $} $  $ as the element of $ K_{0}(\add\relGammabf) $ given by
$$ \ind_{\bmgamma}X=[T_{0}^{X}]-[T_{1}^{X}] ,$$
where $ T_{1}^{X}\ra T_{0}^{X}\ra X\ra\Si T_{1}^{X} $ is an $ (\add\bmgamma) $-presentation of $ X $. It does not depend on the choice of a presentation \cite{paluClusterCharacters2Calabi2008}.

For a dimension vector $ e\in\mathbb{N}^{Q_{0}} $, we denote by $ l(e) $ the sum $ \ind_{\bmgamma}X+\ind_{\bmgamma}\Si X $, where $ \underline{\dim}\Ext^{1}_{\cc}(\bmgamma,X)=e $. By \cite[Lemma 5.2]{KellerWu2023}, this does not depend on the choice of such $ X $. The $G$-action on $(Q,F,W)$ induces a $G$-action on $\mathbb{N}^{Q_{0}}$. 

Let $\pi$ be the following canonical projection
$$\pi\colon\mathbb{Q}[x_{i}^{\pm}]_{i\in I}\ra\mathbb{Q}[x_{\bm{i}}^{\pm}]_{\bm{i}\in\overline{I}}$$
$$\qquad\,\, x_{i}\mapsto x_{\bm i}.$$

\begin{Lem}\label{Lem:G-index}
	\begin{itemize}
		\item For each $M\in\pr_{\cc(Q,F,W)}\bmgamma$, the polynomial $$\pi(x^{\ind_{\bmgamma}(M)})=\pi(\prod_{i=1}^{n}x_{i}^{[\ind_{\bmgamma}(M):\bmgamma_{i}]})$$ only depends on the class of $M$ modulo $G$.
		\item For each $ e\in\mathbb{N}^{Q_{0}} $, the polynomial $\pi(x^{-l(e)})$ only depends on the class of $e$ modulo $G$. 
	\end{itemize}
\end{Lem}
\begin{proof}
	Since $\bmgamma$ is $G$-invariant, for every $g\in G$, we have
	$$[\ind_{\bmgamma}(^{g}\!M):\bmgamma_{i}]=[\ind_{\bmgamma}(M):\bmgamma_{^{g^{-1}}(i)}].$$
	Then it is clear that $\pi(x^{\ind_{\bmgamma}(M)})=\pi(x^{\ind_{\bmgamma}(^{g}\!M)})$ for each $g\in G$. This proves the first statement.
	
	For each $e\in\mathbb{N}^{Q_{0}}$, $g\in G$ and $X\in\ch$ such that $\underline{\dim}\Ext^{1}_{\cc}(\bmgamma,X)=e$, we have $$\Ext^{1}_{\cc}(\bmgamma,^{g}\!X)\simeq\Ext^{1}_{\cc}(^{g}\bmgamma,^{g}\!X)\simeq\,^{g^{-1}}\Ext^{1}_{\cc}(\bmgamma,X).$$
	Hence $\underline{\dim}(\Ext^{1}_{\cc}(\bmgamma,^{g}\!X))=\,^{g^{-1}}\!e$. Therefore $\pi(x^{-l(e)})$ only depends on the class of $e$ modulo $G$.
	
\end{proof}

Define the map (\cite{KellerWu2023})
\begin{align}\label{CC formula for Higgs}
	\xymatrix{
		CC_{-}:\obj(\ch)\ra\mathbb{Q}[x_{r+1},\ldots,x_{n}][x^{\pm1}_{1},x_{2}^{\pm1},\ldots,x_{r}^{\pm1}]
	}
\end{align}
as follows:
for any object $ M $ of $ \ch $, we put
$$ CC_{M}=x^{\ind_{\bmgamma}M}\sum_{e}\chi(\Gr_{e}(\Ext^{1}_{\cc}(\bmgamma,M)))x^{-l(e)}, $$
where the sum ranges over all the elements of the Grothendieck group; for a $ J_{rel} $-module $ L $, the notation $ \Gr_{e}(L) $ denotes the projective variety of submodules of $ L $ whose class in the Grothendieck group is $ e $; for an algebraic variety $ V $ over $ \mathbb{C} $, the notation $ \chi(V) $ denotes the Euler characteristic.

\begin{Thm}\label{Thm: relative cc map}\cite[Theorem 5.4]{KellerWu2023}
	The map $ CC_{-} $ defined above is a cluster character on $ \ch $ with respect to $ \bmgamma $.
\end{Thm}

\begin{Def}\rm\label{Def:P_X}
	For an object $M$ in $\ch$, we define the Laurent polynomial $P_{M}$ of $\mathbb{Q}[x_{\bm{i}}^{\pm}]_{\bm{i}\in\overline{I}}$ by
	$$P_{-}\colon\obj(\ch)\ra\mathbb{Q}[x_{\bm{i}}^{\pm}]_{\bm{i}\in\overline{I}}$$
	$$\qquad\qquad\qquad\, M\mapsto P_{M}=\pi(CC_{M}).$$
	
\end{Def}

\begin{Prop}
The map $P_{-}\colon\obj(\ch)\ra\mathbb{Q}[x_{r+1},\ldots,x_{n}][x^{\pm1}_{1},x_{2}^{\pm1},\ldots,x_{r}^{\pm1}]$ is $G$-equivariant, i.e. for each $M$ of $ \ch$, $P_{M}$ only depends on the class of $M$ modulo $G$.
\end{Prop}
\begin{proof}
It follows from Lemma~\ref{Lem:G-index} and the following isomorphism \cite[Lemma 3.4]{LD2011} $$\Gr_{e}(\Ext^{1}_{\cc}(\bmgamma,^{g}\!M))\simeq\Gr_{^{g}(e)}(\Ext^{1}_{\cc}(\bmgamma,M)).$$

\end{proof}

One will denote
$$P_{\overline{M}}=P_{M},$$
where $\overline{M}$ is the $G$-orbit of $M$.

\begin{Prop}
		\label{thm:3.50}
		\begin{itemize}
			\item[(i)] For $\bm i \in I, P_{\overline{\bmgamma_i}}= x_i$.
			\item[(ii)] If $X, Y \in \mathcal{H}, P_{\overline{X} \oplus\overline{Y}}= P_{\overline{X}} P_{\overline{Y}}$.
			\item[(iii)] If $X, Y \in \mathcal{H}$ and $\dim \mathrm{Ext}^1_{\mathcal{C}}(X, Y) = 1$, and if one fixes two non-split admissible short exact sequences
			\[
			0 \to X \to Z \to Y \to 0 \quad \text{and} \quad 0 \to Y \to Z' \to X \to 0
			\]
			then $P_{\overline{X}} P_{\overline{Y}} = P_{\overline{Z}} + P_{\overline{Z'}}$.
		\end{itemize}

\end{Prop}

\begin{proof}
	This follows from \cite[Theorem 5.4]{KellerWu2023} by applying the ring morphism $\pi$.
\end{proof}

\bigskip

%
%

Denote by $\mathcal{A}( \mathcal{H},G)$ the subalgebra of $\mathbb{Q}(x_i)_{i \in \overline{I}}$ generated by the $P_{\overline{X}}$ where $\overline{X}$ goes over all $G$-orbits of objects of $\mathcal{H}$ such that $\bigoplus_{X \in \overline{X}} X$ is rigid. Denote by $\mathcal{A}( \mathcal{H})$ the subalgebra of $\mathbb{Q}(x_i)_{i \in I}$ generated by the $CC_{X}$ where $X$ goes over the rigid objects of $\mathcal{H}$.

\medskip

Denote by $\mathcal{A}_0(\mathcal{H},G)$ the subalgebra of $\mathbb{Q}(x_i)_{i \in \overline{I}}$ generated by the $P_{\overline{X}}$ where $\overline{X}$ goes over the $G$-orbits of objects of $\mathcal{H}$. Denote by $\mathcal{A}_0( \mathcal{H})$ the subalgebra of $\mathbb{Q}(x_i)_{i \in I}$ generated by the $CC_X$ where $X$ goes over $\mathcal{H}$.

Let $\ca(B(\bmgamma))$ be the cluster algebra with coefficients whose initial seed is $(B(\bmgamma),\{x_{\bm{i}}\,|\,\bm{i}\in\overline{I}\})$ and $\ca(\tilde{B}(\bmgamma))$ the cluster algebra with coefficients whose initial seed is $(\tilde{B}(\bmgamma),\{x_{i}\,|\,i\in I\})$.

\begin{Prop}\cite[Corollary 3.52]{LD2011}
There is a commutative diagram of inclusions
\[
\begin{tikzcd}
	\ca(B(\bmgamma))\arrow[r,hook]\arrow[d,hook]&\pi(\ca(\tilde{B}(\bmgamma))\arrow[d,hook]\\
	\ca(\ch,G)\arrow[r,hook]\arrow[d,hook]&\pi(\ca(\ch))\arrow[d,hook]\\
	\ca_{0}(\ch,G)\arrow[r,equal]&\pi(\ca_{0}(\ch)).
\end{tikzcd}
\]

\end{Prop}

\begin{Ex}(Cluster algebras with principal coefficients in the non-simply laced case)
	
Let $B_{n\times n}$ be a skew-symmetrizable matrix such that the the corresponding Cartan matrix has type of $B_{n}$, $C_{n}$, $F_{4}$ or $G_{2}$. Let $\tilde{B}=\begin{bmatrix}
	B\\
	I_{n\times n}
\end{bmatrix}$. Let $\ca^{\mathrm{prin}}_{B}$ be the principal coefficient cluster algebra such that the initial matrix is $\tilde{B}$. There exist an ice quiver $(Q,F)$ and a finite group $G$, c.f. Example \ref{Ex:type A3}, such that (see \cite[Section 4.2]{LD2011})
\begin{itemize}
	\item the quiver $Q$ is acyclic and $G$ acts on $(Q,F)$ such that $Q$ has no arrows between any two vertices of the same orbit. The $G$-invariant non-degenerate potential $W$ on $Q$ is the zero potential. 
	\item the corresponding skew-symmetrizable matrix defined in Definition \ref{Def:skew-symmetrizable} is $\tilde{B}$.
\end{itemize}

By \cite[Proposition 2.22, Theorem 2.23]{Dupont2009} (or \cite[Theorem 4.47]{LD2011}, \cite[Theorem 6.7]{Huang-Li}), the pair $(Q, G)$ is a stable admissible pair. We obtain the Higgs category $\ch(Q,F)$ which equipped with a $G$-action and the canonical cluster-tilting object $\bmgamma(Q,F)$ is $G$-stable. And the map 
$$P_{-}\colon\ch(Q,F)\ra\mathbb{Q}[x_{1}^{\pm},\cdots,x_{2n}^{\pm}]$$ and the $G$-orbit mutations provide an additive categorification of $\ca^{\mathrm{prin}}_{B}$. Notice that the corresponding relative Ginzburg dg algebra is not concentrated in degree zero in this case. We can still categorify cluster algebras with principal coefficients in the non-simply laced case by using a group action on a Frobenius extriangulated category.

\end{Ex}

\begin{Ex}
Let $Q$ be a finite quiver such that the underlying unoriented graph is a Dynkin diagram of type $A$, $D$ or $E$. Let $G$ be
a group acting on $Q$ in such a way that $Q$ has no arrow between any two vertices of the same orbit, c.f, \cite[Section 4.2]{LD2011}. Consider the following inclusion of algebras
$$f\colon kQ\hookrightarrow \mathrm{Aus}(kQ),$$
where $\mathrm{Aus}(kQ)$ is the Auslander algebra of $kQ$. It is clear that $G$ acts on $f$.

There exists an ice quiver with potential $(Q',F',W')$ such that the relative $3$-Calabi--Yau completion $\tilde{f}\colon\Pi_2(kQ)\ra\Pi_{3}(\mathrm{Aus}(kQ),kQ) $ of $f$ is isomorphic to the Ginzburg functor
$$\bm{G}_{rel}\colon\bm\Pi_2(kF')\ra\bmgamma(Q',F',W').$$
Moreover, the group $G$ acts on $(Q',F',W')$, i.e. the Assumption \ref{ass1} is satisfied and $W'$ is non-degenerate, see \cite{GLS2006,Wu2023}. In Example \ref{Ex:typeA5}, the $\mathbb{Z}/2\mathbb{Z}$-invariant non-degenerate potential can be chosen as the alternating sum of the triangles in $Q$.

By \cite[Corollary 8.7]{Wu2023}, the relative Ginzburg algebra $\bmgamma(Q',F',W')$ is concentrated in degree 0. The corresponding Higgs category is equivalent to $\mod(\Pi_{2}(kQ))$, i.e. the module category of the preprojective algebra of $kQ$ \cite[Theorem 8.17]{Wu2023}.
The $G$-stable cluster-tilting object in $\mod(\Pi_{2}(kQ))$ (or the $\mod k[G]$-stable cluster-tilting subcategory of the $G$-equivariant category $\mod(\Pi_{2}(kQ))^{G}$) and the map $P_{-}$ together provide an additive categorification of the cluster structure on the function algebras $\mathbb{C}[N]$, where $N$ is a maximal unipotent subgroup of a simple Lie group in the non-simply-laced case, as studied by Demonet; see \cite{LD20080,LD2008,LD2011}.

\end{Ex}


%
%

\begin{Ex}
Let $\Sigma$ be a topological surface with boundary. Let $M$ be a finite set of marked points on the boundary of $\Sigma$ such that there is at least one marked point on each boundary component of $\Sigma$. Let $P$ be a finite set of marked points in the interior of $\Sigma$, called punctures. Assume that$\colon$
\begin{itemize}
	\item the set of punctures $P$ is non empty;
	\item $(\Sigma,M,P)$ is not a once-punctured monogon.
\end{itemize}

Let $\tau$ be an ideal triangulation of $(\Sigma,M,P)$ of $\Sigma$ (in the sense of \cite[Definition 2.6]{Fomin-Shapiro-Thurston2008}) such that each puncture belongs to a self-folded triangle and such that no triangle shares a side with two self-folded triangles.

There are exactly six different types of triangles in $\tau$ which are not self-folded \cite{Fomin-Shapiro-Thurston2008}\cite[Section 3.1]{Amiot-Plamondon2018}. The corresponding adjacency ice quiver $(Q(\tau),F(\tau))$ is built by gluing blocks corresponding to each kind of triangle.

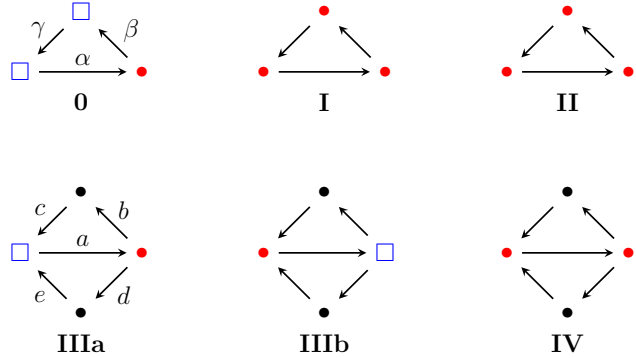
\[\scalebox{0.8}{
	\begin{tikzpicture}[>=stealth,scale=1]
		
		\node (I0) at (-5,0) {$\color{blue}\square$};
		\node (J0) at (-3,0) {$\color{red}\bullet$};
		\node (K0) at (-4,1) {$\color{blue}\square$}; 
		\draw[thick, ->] (I0)--node [yshift=2mm]{$\alpha$}(J0);
		\draw[thick, ->] (J0)--node [yshift=2mm,xshift=3mm]{$\beta$}(K0);
		\draw[thick, ->] (K0)--node [yshift=2mm,xshift=-2mm]{$\gamma$}(I0);
		\node at (-4,-0.5) {$\mathbf{0}$};

		
		\node (I0) at (-1,0) {$\color{red}\bullet$};
		\node (J0) at (1,0) {$\color{red}\bullet$};
		\node (K0) at (0,1) {$\color{red}\bullet$}; 
		\draw[thick, ->] (I0)--(J0);
		\draw[thick, ->] (J0)--(K0);
		\draw[thick, ->] (K0)--(I0);
		\node at (0,-0.5) {$\mathbf{I}$};

		
		\node (I0) at (3,0) {$\color{red}\bullet$};
		\node (J0) at (5,0) {$\color{red}\bullet$};
		\node (K0) at (4,1) {$\color{red}\bullet$}; 
		\draw[thick, ->] (I0)--(J0);
		\draw[thick, ->] (J0)--(K0);
		\draw[thick, ->] (K0)--(I0);
		\node at (4,-0.5) {$\mathbf{II}$};

		\node (I3) at (-5,-3) {$\color{blue}\square$};
		\node (J3) at (-3,-3) {$\color{red}\bullet$};
		\node (K3) at (-4,-2) {$\bullet$};
		\node (L3) at (-4,-4) {$\bullet$}; 
		\draw[thick, ->] (I3)--node [yshift=2mm]{$a$}(J3);
		\draw[thick, ->] (J3)--node [yshift=2mm,xshift=2mm]{$b$}(K3);
		\draw[thick, ->] (K3)--node [yshift=2mm,xshift=-2mm]{$c$}(I3);
		\draw[thick, ->] (J3)--node [yshift=-2mm,xshift=2mm]{$d$}(L3);
		\draw[thick, ->] (L3)--node [yshift=-2mm,xshift=-2mm]{$e$}(I3);
		\node at (-4,-4.5) {$\mathbf{IIIa}$};
		
		\node (I3) at (3,-3) {$\color{red}\bullet$};
		\node (J3) at (5,-3) {$\color{red}\bullet$};
		\node (K3) at (4,-2) {$\bullet$};
		\node (L3) at (4,-4) {$\bullet$}; 
		\draw[thick, ->] (I3)--(J3);
		\draw[thick, ->] (J3)--(K3);
		\draw[thick, ->] (K3)--(I3);
		\draw[thick, ->] (J3)--(L3);
		\draw[thick, ->] (L3)--(I3);
		\node at (0,-4.5) {$\mathbf{IIIb}$};
		\node (I3) at (-1,-3) {$\color{red}\bullet$};
		\node (J3) at (1,-3) {$\color{blue}\square$};
		\node (K3) at (0,-2) {$\bullet$};
		\node (L3) at (0,-4) {$\bullet$}; 
		\draw[thick, ->] (I3)--(J3);
		\draw[thick, ->] (J3)--(K3);
		\draw[thick, ->] (K3)--(I3);
		\draw[thick, ->] (J3)--(L3);
		\draw[thick, ->] (L3)--(I3);
		\node at (4,-4.5) {$\mathbf{IV}$};
\end{tikzpicture}}\]
Any two triangles can only be glued by identifying two vertices of type $\color{red}\bullet$ and one block cannot be glued to itself. The non-degenerate potential $W(\tau)$ \cite{Labardini2009} associated to $\tau$ is defined as 
$$W(\tau)=\sum_{\text{blocks of type 0,I,II}}\gamma\beta\alpha+\sum_{\text{blocks of type IIIa,IIIb,IV}}(cba+eda).$$
Hence we obtain an ice quiver with potential $(Q(\tau),F(\tau),W(\tau))$.

In \cite[Section 3]{Amiot-Plamondon2018}, Amiot--Plamondon constructed a new unpunctured marked surface $(\tilde{\Sigma},\tilde{M})$ together with a triangulation $\tilde{\tau}$ such that the associated ice quiver with potential $(\tilde{Q},\tilde{F},\tilde{W})$ has a $G=\mathbb{Z}/2\mathbb{Z}$-action. Moreover, the corresponding ice quiver with potential $(\tilde{Q}_{G},\tilde{F}_{G},\tilde{W}_{G})$ defined in Definition \ref{Def: Q_{G} and F_{G}} is right equivalent to $(Q(\tau),F(\tau),W(\tau))$ \cite[Theorem 3.5]{Amiot-Plamondon2018}.

By Theorem \ref{Prop: skew-Higgs}, we have the following equivalence of Frobenius extriangulated categories
$$\ch(Q(\tau),F(\tau),W(\tau))\simeq\ch(\tilde{Q},\tilde{F},\tilde{W})^{\mathbb{Z}/2\mathbb{Z}}.$$ 


\end{Ex}

%


\bibliographystyle{plain}

\end{document}